\documentclass[11pt]{amsart}
\usepackage{geometry}                
\usepackage{graphicx}
\usepackage{amssymb}
\usepackage{epstopdf}
\usepackage{caption}
\usepackage{subcaption}
\usepackage{appendix}
\newtheorem{theorem}{Theorem}[section]

\newtheorem{corollary}[theorem]{Corollary}
\newtheorem{definition}[theorem]{Definition}

\newtheorem{lemma}[theorem]{Lemma}

\newtheorem{proposition}[theorem]{Proposition}

\newtheorem*{thm:edges}{Theorem \ref{thm:edges}}
\newtheorem*{thm:triangle}{Theorem \ref{thm:triangle}}
\newtheorem*{cor:triangle}{Corollary \ref{cor:triangle}}

\newcommand{\GarLe}{Garoufalidis and L\^{e} }

\title{Higher Order Stability in the Coefficients of the Colored Jones Polynomial}
\author{Katherine Walsh}
\address{Department of Mathematics, University of Arizona, 
Tucson, AZ 85721}
\email{k3walsh@math.arizona.edu}

\begin{document}
\maketitle

\begin{abstract}
We discuss the higher order stabilization of the coefficients of the colored Jones polynomial. In particular, we find an expression for the second stable sequence of the colored Jones polynomial of a certain class of knots.  We also determine which knots have the same higher order stability. 
\end{abstract}


\section{Introduction}
The colored Jones polynomial is a knot invariant that assigns to each knot $K$ a sequence of Laurent polynomials $\{J_{N,K}(q)\}.$ An important open question is how to relate the colored Jones polynomial of a knot to the knot's geometry.  One such relation is the hyperbolic volume conjecture of Kashaev, Murakami and Murakami \cite{MurVol}, which states that one can find the hyperbolic volume of a knot's complement by evaluating $J'_{N,K}(q)$ (the appropriately normalized colored Jones polynomial) at an $N$th root of unity and then taking a particular limit as $N$ goes to infinity. Specifically,  

 \begin{equation}
 2\pi\lim_{N \rightarrow \infty} \frac{\log |J'_{N,K}(e^{2\pi i/N})|}{N}=\text{vol}(S^3\backslash K).
 \end{equation} 

The hyperbolic volume conjecture has been proven for various knots, including the figure-eight knot, (see \cite{MurVol}) but is still open for many knots and links.

In \cite{Volish}, Dasbach and Lin related the first and last two coefficients of the original Jones polynomial of alternating, prime, non-torus knots to the the volume of the knot in the following way: Let
\begin{equation}
J'_{2,K}(q)= a_nq^n+ \cdots + a_m q^m
\end{equation} be the Jones polynomial of $K$. Then 
\begin{equation}
2v_8(\textrm{max}(|a_{m-1}|,|a_{n+1}|)-1) \leq \textrm{vol}(S^3 \backslash K) \leq 10v_3(|a_{n+1}|+|a_{m-1}|-1).
\end{equation}
Here, $v_3 \approx 1.0149416$ is the volume of an ideal regular hyperbolic tetrahedron and $v_8 \approx 3.66386$ is the volume of an ideal regular hyperbolic octahedron.

Dasbach and Lin also proved that the first two and last two coefficients of the Jones polynomial were the same as the first and last two coefficients of the $N$ colored Jones polynomial for all $N$ and noticed that the first and last $N$ coefficients of the $N$ colored Jones polynomial were the same, up to sign, as the first $N$ coefficients of the $k$ colored Jones polynomial for all $k>N$.   These theorems and observed patterns encourage us to look more deeply at the coefficients of the colored Jones polynomial to see what they can tell us about the knot.

Given a sequence of Laurent polynomials, the \emph{head} of this sequence exists if the highest $N$ coefficients (coefficients of the $N$ terms with the highest $q$ degree) of the $N${th} polynomial in the sequence are the same as the highest $N$ coefficients of the $k${th} polynomial for all $k\ge N$.  The \emph{tail} of the sequence of polynomials, if it exists, is the stabilized sequence of the coefficients of the lowest terms.

In \cite{Arm, DasArm}, Dasbach and Armond proved that the head and tail of the colored Jones polynomial exist for alternating and adequate knots and depend on the reduced checkerboard graphs of the knot diagrams.  
%

 For example, they show that for any knot whose B-checkerboard graph is a triangle graph, the highest coefficients stabilize to the pentagonal number sequence (expanded so that we have a maximum degree of 0). By this, we mean that for these knots, the highest $N$ coefficients of $J'_{N,K}$ are the same as the highest $N$ coefficients of  
 \begin{equation}
 T_0=\prod_{n=1}^\infty(1-q^{-n})=\sum_{k=-\infty}^{\infty}(-1)^kq^{\frac{-k}{2}(3k-1)}.
 \end{equation}

The knot $8_5$ has an A-checkerboard graph that is a triangle graph. This means that the mirror image of $8_5, \overline{8_5}$ will have a B-checkerboard graph which is a triangle graph. The table below lists the highest several coefficients of the N-colored Jones polynomial for the knot $ \overline{8_5}$ for $N=5, 6$ and 7. See the appendix for the calculation of the first $3N-2$ of these, the rest can be found on the Knot Atlas \cite{KnotAtlas}. We see that the highest $N$ coefficients of the $N$ colored Jones polynomial are the same as the highest $N$ coefficients of $T_0$.\\
\begin{table}[h!]
\begin{tabular}{lcccccccccccccccccccccc}
\hline
$T_0$	&1	& -1	& -1	& 0	&0	&1	& 0	& 1	& 0	& 0	& 0	& 0	& -1	& 0	& 0	&	$\cdots$ \\ 
\hline
$N=5$ 	&1	& -1	& -1	& 0	& 0	& 5	& -1	&  -3	&  -3	& -5	&11 & 4 &1 &-6 &17 &$\cdots$ \\ 
$N=6$ 	&1	& -1	& -1	& 0	& 0	& 1	& 4	& 0	& -4	& -3	& -3	& -1		& 9	& 8	& 1	&$\cdots$ \\ 
$N=7$ 	&1	& -1	& -1	& 0	& 0	& 1	& 0	&  5	& -1	& -4	& -3	& -3		&  0	&  -2	& 14	&	$\cdots$ \\ 
 \hline
\end{tabular}
 \end{table}

 Now, since we know all of $T_0$, we can subtract it from the shifted colored Jones polynomials.  Now the coefficients are:\\

\begin{table}[h!]
\begin{tabular}{lccccccccccccccccccccc}
\hline
$T_0$	&1	& -1	& -1	& 0	&0	&1	& 0	& 1	& 0	& 0	& 0	& 0	& -1		& 0	& 0	& $\cdots$ \\
\hline
$N=5$ 	&0	& 0	& 0	& 0	& 0	& 4	& -1	&-4	&  -3	& -5 & 11 &4 &2&-6 &17 &$\cdots$ \\ 
$N=6$ 	&0	& 0	& 0	& 0	& 0	& 0	& 4	& -1	& -4	& -3	& -3	& -1	& 10		& 8	& 1	& $\cdots$\\
$N=7$ 	&0	& 0	& 0	& 0	& 0	& 0	& 0	&  4	& -1	& -4	& -3	& -3	&  1		& -2	& 14 & $\cdots $\\
 \hline

\end{tabular}
\end{table}

Shifting these sequences back so that they start with a non-zero term, we can see that they again stabilize, but now only $N-1$ terms stabilize. We call the sequence they stabilize to $T_1$.\\

\begin{tabular}{lccccccccccccccccccc}
\hline
$T_1$	&4	& -1	& -4	& -3	&-3	&1	& 0	& 4	& 3	& 3	&$\cdots$ \\
\hline
$N=5$ 	& 4	& -1	&-4	&  -3	& -5	& 11	& 4	&2&-6&17&$\cdots$ 		& 	&\\
$N=6$ 	& 4	& -1	& -4	& -3	& -3	& -1	& 10		& 8	& 1	&-4 & $\cdots$ 	&\\
$N=7$ 	& 4	& -1	& -4	& -3	& -3	&  1		& -2	& 14 &7&1&	$\cdots$ \\
 \hline

\end{tabular}\\

After another subtraction and shifting (by $N-1$), we see that this pattern continues:\\

\begin{tabular}{lccccccccccccccccccc}
\hline
$T_2$	&-2	& 10	& 4	& -2	&-7	&-12&$\cdots$ \\
\hline
$N=5$ 	& -2	& 10	&4&-2&-9&14&$\cdots$ \\
$N=6$ 	& -2	& 10	& 4	& -2	&-7&-14&$\cdots$ \\
$N=7$ 	& -2	& 10	& 4	& -2	& -7	&-12&$\cdots$ \\
 \hline

\end{tabular}\\

Following a suggestion by Dasbach, we can also define this higher order stability another way. Instead of subtracting off the stabilized sequence, we can find $J_{N,K}-q^*J_{N+1,K}$, where the power of $q$ ensures that these polynomials have the same degree. Then with these differences, we can find the second order differences, and so forth. These sequences of differences also stabilize. We will show that their stability is equivalent to the stability we see above.

We call the sequence $T_1$ the ``neck" of the colored Jones polynomial. (The corresponding stable sequence in the lower degree terms is called the ``tailneck.") The main results in this paper relate to determining which knots will have the same higher order stability, or equivalently, what reductions can be done to a knot which do not change its higher order stability, and explicitly finding $T_1$ for three strand pretzel knots with negative twists. In particular, the main theorems are stated below.

Given a B-checkerboard graph of a knot, we might have multiple parallel edges corresponding to multiple negative twists in a region. We get a reduced graph by reducing all parallel edges down to a single edge.  We get an $m$-reduced graph by reducing $m+1$ or more parallel edges down to $m$ parallel edges.

\begin{theorem}
\label{thm:edges}
If $K_1$ and $K_2$ are alternating knots whose alternating diagrams have the same $m+1$-reduced B-checkerboard graph structure, then the highest (m+1)N coefficients (in $q$) of $J_{N,K_1}$ are the same as the first $(m+1)N$ coefficients of $J_{N,K_2}$, up to possible rescaling by $\pm 1$.
\end{theorem}

 \begin{figure}[ht]
\centering
 \includegraphics[width=.33\textwidth]{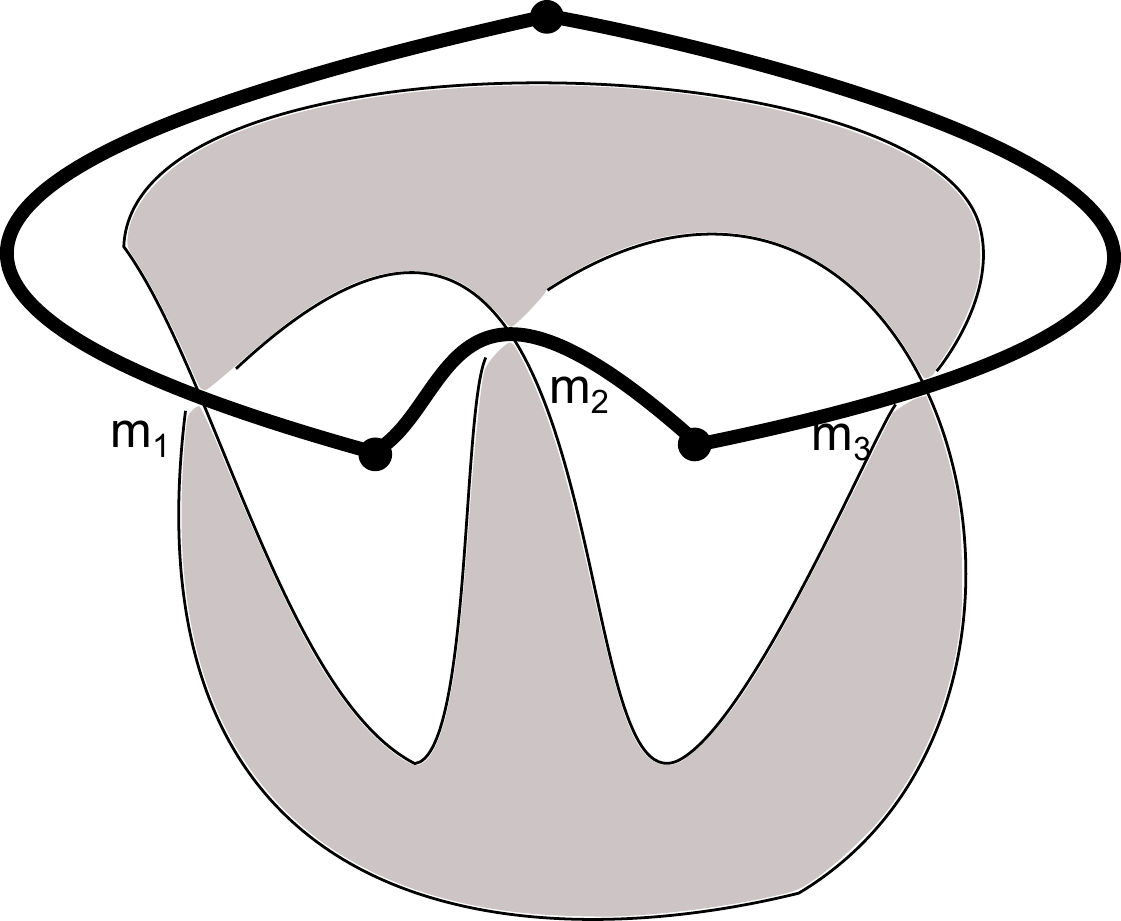}
  \caption{A trefoil knot with its checkerboard graph.}
 \label{treftail1}
 \end{figure}

Consider three strand pretzel knots with negative twists in each region. For knots in this family, the B-checkerboard graph is a triangle graph. These knots can be drawn like the trefoil in Fig. \ref{treftail1}, except we will have more crossings below the pictured crossings  (and thus more parallel edges before we reduce the graph).  Let $m_i$ represent the number of crossings in each section.

\begin{theorem}
  \label{thm:triangle}

  Let $m$ be the number edges in the checkerboard graph with $m_i$ of 2 or more.   The neck of knots whose reduced checkerboard graph is the triangle graph is:
\begin{equation} \prod_{n=1}^\infty(1-q^{-n})+m\frac{\prod_{n=1}^\infty(1-q^{-i})}{1-q^{-1}}, \end{equation}
i.e. the pentagonal numbers plus the $m$ times the partial sum of the pentagonal numbers.
  \end{theorem}

  Using the alternate view of stability provided by Dasbach, we get the following corollary:

 \begin{corollary}
 Again, let $m$ be the number edges in the checkerboard graph with $m_i$ of 2 or more.  Then we have
 \begin{equation}J'_{N,K}-q^*J'_{N+1,K}\stackrel{\cdot N-1}{=}(1+m-q^{-1})\prod_{n=1}^\infty(1-q^{-n}). \end{equation}
 \end{corollary}

  The paper is arranged as follows: In Section 2, we will show how the existence of $T_m$ is similar to a property proved by \GarLe.  We then explicitly compute this sequence, in certain cases, using the framework set up by Armond and Dasbach.  In particular,  in Section 3, we show what graph reductions can be done to find these stable sequences. Then, in Section 4, we find an explicit description of $T_1$ for knots whose B-checkerboard graph is a triangle graph.  
  
 \section*{Acknowledgments}
This paper is largely based on work I did for my PhD under my advisor, Justin Roberts. I would like to thank him for our many, many helpful conversations over the years. I am also thankful for various conversations with Cody Armond and Oliver Dasbach about their work relating to stability of the colored Jones sequence. Thanks to the referee for helpful suggestions and clarifying remarks.

\section{Connection to the work of \GarLe}

The motivation for looking at this higher order stability came from the work of \GarLe. In \cite{Nahm}, they independently proved that the head and tail of the colored Jones polynomial exist for alternating knots while proving (for alternating knots) a stronger version of this stability. In particular, they defined the property of $k$-stability for a sequence of polynomials as follows:
\begin{definition}
Suppose $f_n(q), f(q) \in \mathbb{Z}((q)),$ i.e. $f_n(q)$ and $f(q)$ are formal Laurent series -- series that can be written as $\sum_{i\ge m}a_iq^i$ where $a_i \in \mathbb{Z}$.  We write that
\begin{equation} \lim_{n \to \infty} f_n(q) =f(q) \end{equation}
if
\begin{itemize}
\item there exists $C$ such that $\textrm{mindeg}_q(f_n(q)) \geq C$ for all $n$, and
\item 
for each $j \in \mathbb{Z}$, 
\begin{equation} \lim_{n \to \infty} \textrm{coeff}(f_n(q),q^j)=\textrm{coeff}(f(q),q^j) \end{equation}
This implies that for each $j$ there exist $N_j$ such that for all $n > N_j$
\begin{equation}f_n(q)-f(q) \in q^j\mathbb{Z}[[q]]\end{equation}
In particular, $\textrm{coeff}(f_n(q),q^j)=\textrm{coeff}(f(q),q^j)$ for all $n>N_j$.
\end{itemize}
\end{definition}
\begin{definition}
A sequence ($f_n(q)) \in \mathbb{Z}[[q]]$ is \emph{k-stable} if there exist $\Phi_j(q) \in \mathbb{Z}((q))$ for $j=0,\ldots,k$ such that
\begin{equation}\lim_{n \to \infty} q^{-k(n+1)}\left(f_n(q)-\sum_{j=0}^k \Phi_j(q)q^{j(n+1)}\right)=0.\end{equation}
We call $\Phi_k(q)$ the $k$-limit of $(f_n(q))$. We say that $(f_n(q))$ is \emph{stable} if it is $k$-stable for all $k$.
\end{definition}
For example, a sequence $(f_n(q))$ is \emph{3-stable} if

\begin{equation}\lim_{n \to \infty} q^{-3(n+1)}\left(f_n(q)-\left(\Phi_0(q)+q^{(n+1)}\Phi_1(q)+q^{2(n+1)}\Phi_2(q)+q^{3(n+1)}\Phi_3(q)\right)\right)=0.\end{equation}

The property of the head and tail existing for the colored Jones Polynomial of a knot  is similar to the condition of the colored Jones sequence being $0-$stable. In addition, in \cite{Nahm}, \GarLe proved the following theorem about higher order stability.
\begin{theorem}[\cite{Nahm}]
\label{GarThm}
For every alternating link $K$, the sequence $f_N(q)=(\hat{J}_{N+1,K}(q))$ is stable and its associated $k$-limit $\Phi_{K,k}(q)$ can be effectively computed from any reduced alternating diagram $D$ of $K$.
\label{highstab}
\end{theorem}

A note on indexing, scaling and normalization: \GarLe use the convention that $J_{n,K}$ gives the colored Jones polynomial with each component colored by the ($n+1$)-dimensional irreducible representation of $\mathfrak{sl}_2$. (So, $n=1$ gives the standard Jones polynomial.)  In this paper, we use another standard convention that instead $J_{N,K}$ gives that colored Jones polynomial with each component colored by the $N$-dimensional irreducible representation.  (So, $N=2$ gives the standard Jones polynomial). These changes were already made to the statement of Theorem \ref{GarThm} above. To obtain $\hat{J}_{N+1,K}$, we divide $J_{N+1,k}$ by its lowest monomial so that its lowest term is now 1. These polynomials are not normalized. We generally normalize so that the value of the colored Jones polynomial of the unknot is 1.  We will call this normalized polynomial, $J'_{N,K}.$  In \cite{Nahm}, \GarLe remark that the same stability holds for the normalized colored Jones sequence. 

Instead of viewing stability as a limit, we look at stability from the beginning of the sequence by focusing on particular coefficients of the colored Jones polynomials in the colored Jones sequence for a given knot.  First, we will consider the similarities and differences between the theorem proved in \cite{Nahm} and the pattern seen in the subtraction of coefficients. 

The theorem tells us that the colored Jones sequence for alternating knots is stable, meaning it is $k$-stable for any $k$. To begin, let's look at what it means for the sequence to be $0$-stable. If $f_N(q)=J_{N+1,K}$ is 0-stable then 
\begin{equation}\lim_{N \to \infty} J_{N+1,K}(q)-\Phi_0(q)=0.\end{equation}
This only means that $J_{N+1,K}$ matches $\Phi_0(q)$ for an arbitrarily large number of terms for sufficiently large $N$.  It does not guarantee that the first $N+1$ terms of $J_{N+1,K}(q)$ are the same as those of  $\Phi_0(q)$. 

However, as we continue on and look at $1$-stability, we have that  
\begin{equation}\lim_{N \to \infty} q^{-(N+1)} \left(J_{N+1,K}(q)-\Phi_0(q)-q^{N+1}\Phi_1(q)\right)=0.\end{equation} 
In order for the left hand side to have a minimum degree, as the first condition in the definition requires, at some point the first $N+1$ coefficients of $J_{N+1,K}(q)$ must match those of $\Phi_0(q)$ and continue matching as $N$ increases.  Thus, at some point, we must have the first $N+1$ coefficients stabilize.  This is not as strong as the theorem due to Armond in \cite{Arm} since that theorem guarantees this coefficient stabilization from the beginning.  However, the proof given for $0$-stability in the paper  \cite{Nahm}, does include this stronger result. 

As we look at higher order stabilizations, this pattern continues. In particular, a $k+1$-stable sequence must have the property that  
\begin{equation}\lim_{N \to \infty} q^{-(k+1)(N+1)} \left(J_{N+1,K}(q)-\sum_{j=0}^{k+1}q^{j(N+1)}\Phi_j(q)\right)=0.\end{equation} We can see that in order for this to have a lowest degree we must have the the first $k(N+1)$ coefficients of $J_{N+1,K}(q)$ match $\sum_{j=0}^k q^{j(N+1)} \Phi_j$ for large enough $N$.  This gives us, that for large enough $N$, we have the property we observed above, i.e. that $J_{N+1,K}=\Phi_0+q^{N+1}\Phi_1+q^{2(N+1)}\Phi_2+\cdots$.  It, again, however,  does not guarantee this property from the beginning. 

As a note, it seems there may be an indexing error in the statement of the theorem. As we can see in the example in the introduction, for the knot $8_5$, the second stable sequence stabilizes with length one less than guaranteed by the theorem of \GarLe.  This continues as we increase $N$, contradicting their theorem. As a consequence of this discrepancy, we do not rely on this theorem.  The needed results are proved independently using skein theoretical techniques. We also prove that the stability is a property true from the beginning of the sequence, not starting at some unknown $N$.  Because of the $N$ versus $N+1$ discrepancy, we will use different notation, i.e. $T_i$ instead of $\Phi_i(q)$ for the $i$th stable sequence. Also, $T_i$ gives us the highest order terms while $\Phi_i$ gives the lowest order - since we can change $q$ to $q^{-1}$ in the polynomial by taking a mirror image of the knot, this is not a major difference. What we will show is a special case of the results of \GarLe with a change of $N$ to $N+1$ in specific cases. 

\section{Finding $T_i(q)$ from the reduced graph.}

In \cite{DasArm}, Armond and Dasbach show that the head and tail of the colored Jones polynomial of alternating links only depend on the reduced checkerboard graphs of the knot diagrams.  We will show that to find $T_m(q)$, we can reduce any parallel edges of $m+1$ or more in the checkerboard graph with $m+1$ parallel edges. We call this the $m+1$-reduced checkerboard graph.

 Given an alternating diagram of a knot, we can assign a (gray/white) checkerboard coloring the faces in the diagram. We then place a vertex in each of the gray colored regions.  We draw an edge between vertices for every crossing between the corresponding regions.  Alternatively, we can start by placing a vertex in every white region to get the dual graph.  If, when moving along an edge, the over-crossing starts on the right of the edge and ends on the left, this graph is the A-checkerboard graph. If the over-crossing goes from the left to the right, the graph is the B-checkerboard graph.  See Fig. \ref{knotgraph}. To get the $m$-reduced checkerboard graph, we can replace higher order parallel edges in the graph, i.e. $m$ or more edges between the same vertices, with a $m$ parallel edges.  When $m=1$, the 1-reduced graph is the same as the standard reduced graph.

\begin{figure}[ht]
\centering
\begin{subfigure}[b]{.25\textwidth}
\centering
\includegraphics[width=\textwidth]{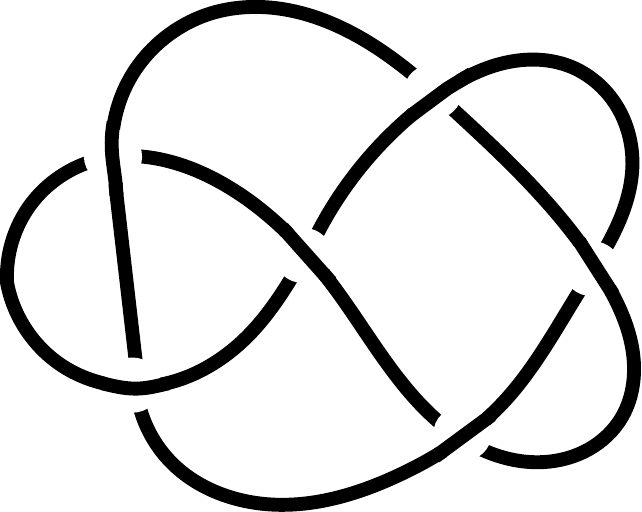}
\caption{A diagram of $6_1$}
\end{subfigure}
\qquad
\begin{subfigure}[b]{.25\textwidth}
\centering
\includegraphics[width=\textwidth]{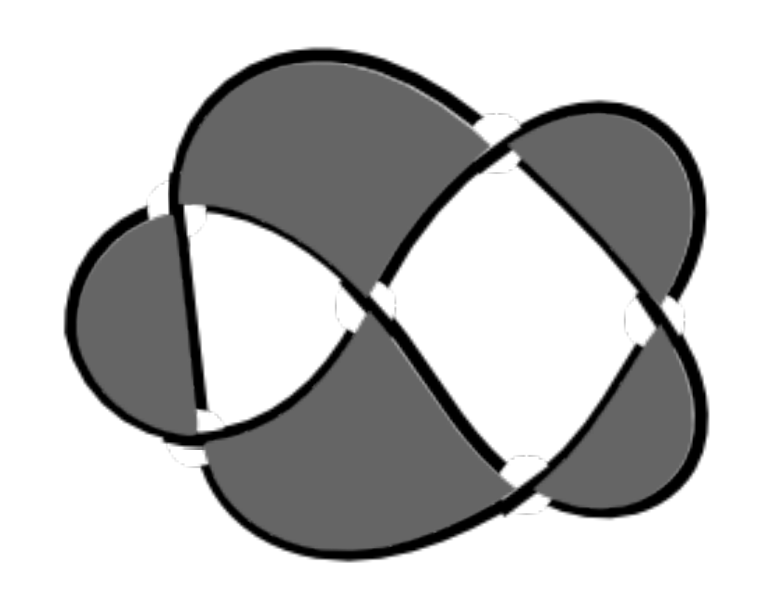}
 \caption{$6_1$ with a checkerboard coloring }
\end{subfigure}

\begin{subfigure}[b]{.25\textwidth}
\centering
\includegraphics[width=\textwidth]{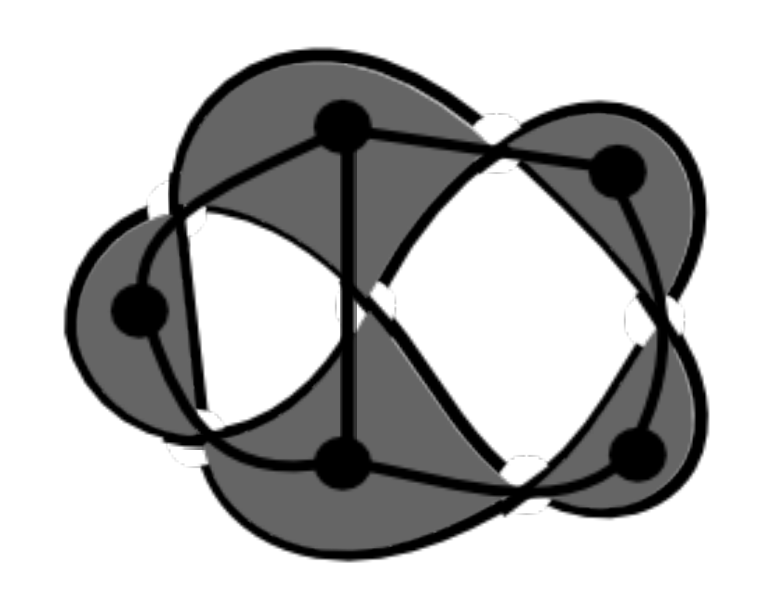}
 \caption{The A-checkerboard graph }
\end{subfigure}
\qquad
\begin{subfigure}[b]{.25\textwidth}
\centering
\includegraphics[width=\textwidth]{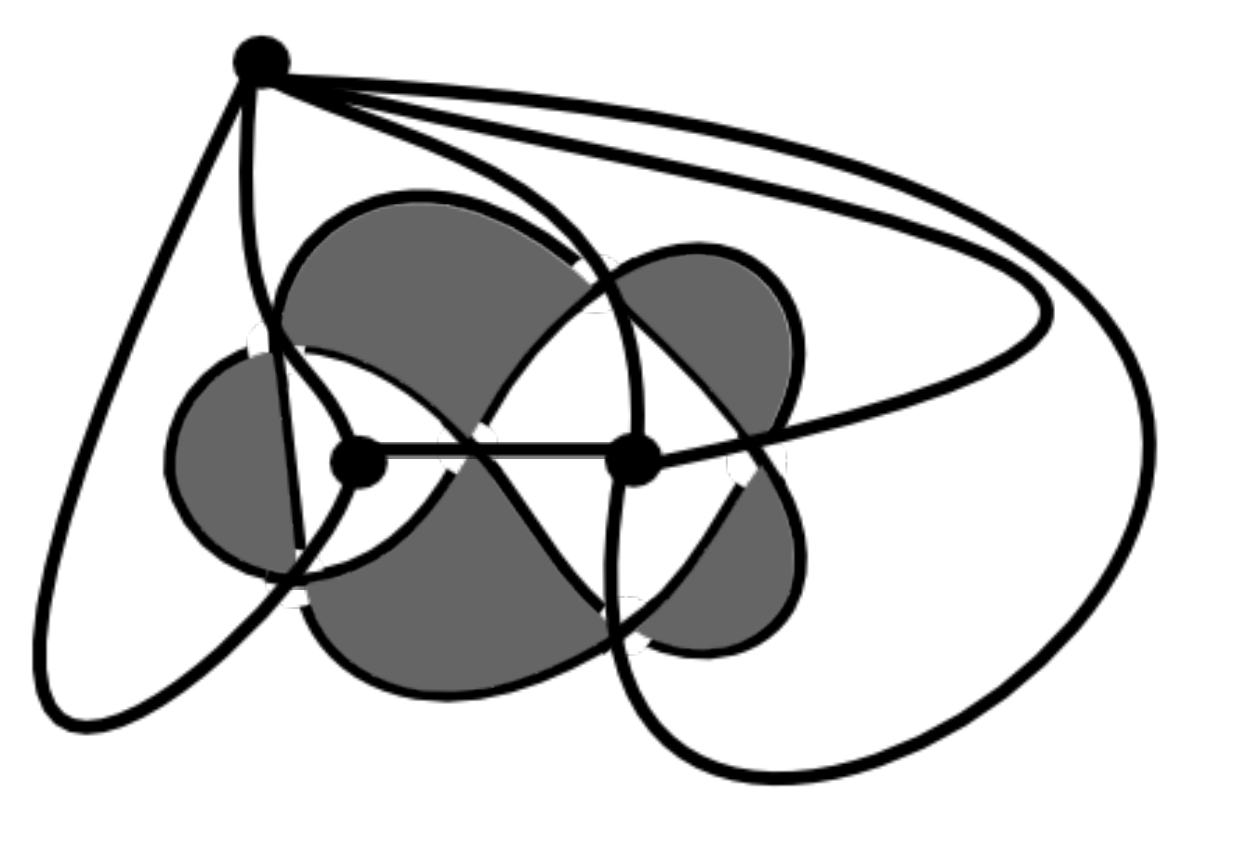}
 \caption{The B-checkerboard graph}
\end{subfigure}
 \caption{The Knot $6_1$ and its associated graphs.}
\label{knotgraph}
\end{figure}

Armond and Dasbach proved that the head and tail of the colored Jones polynomial of alternating links only depend on the reduced graph structure. 

\begin{theorem}[\cite{DasArm}]
Let $K_1$ and $K_2$ be the two alternating links with alternating diagrams $D_1$ and $D_2$ such that the reduced A-checkerboard (respectively B checkerboard) graphs of $D_1$ and $D_2$ coincide. Then the tails (respectively heads) of the colored Jones polynomial of $K_1$ and $K_2$ are identical.
\label{tailstab}
\end{theorem}


The main idea of their proof of this theorem is to show that the head only depends on the highest term in the summand that gives the colored Jones polynomial. To get that sum, we can use formulas for the evaluation and simplification of various pieces.  These formulas can be found in many places, including \cite{Mas}. 



 To find the colored Jones polynomial, we will need to use fusion. 
 
\begin{equation}\raisebox{-.5\height}{\includegraphics[width=.1\textwidth]{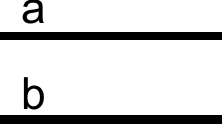}}=\sum_c\frac{\Delta_c}{\theta(a,b,c)}\raisebox{-.5\height}{\includegraphics[width=.1\textwidth]{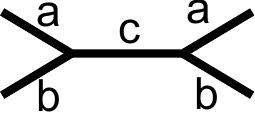}}\end{equation}

This relies on 

\begin{equation}\Delta_n=\left \langle \raisebox{-.5\height}{\includegraphics[scale=.1]{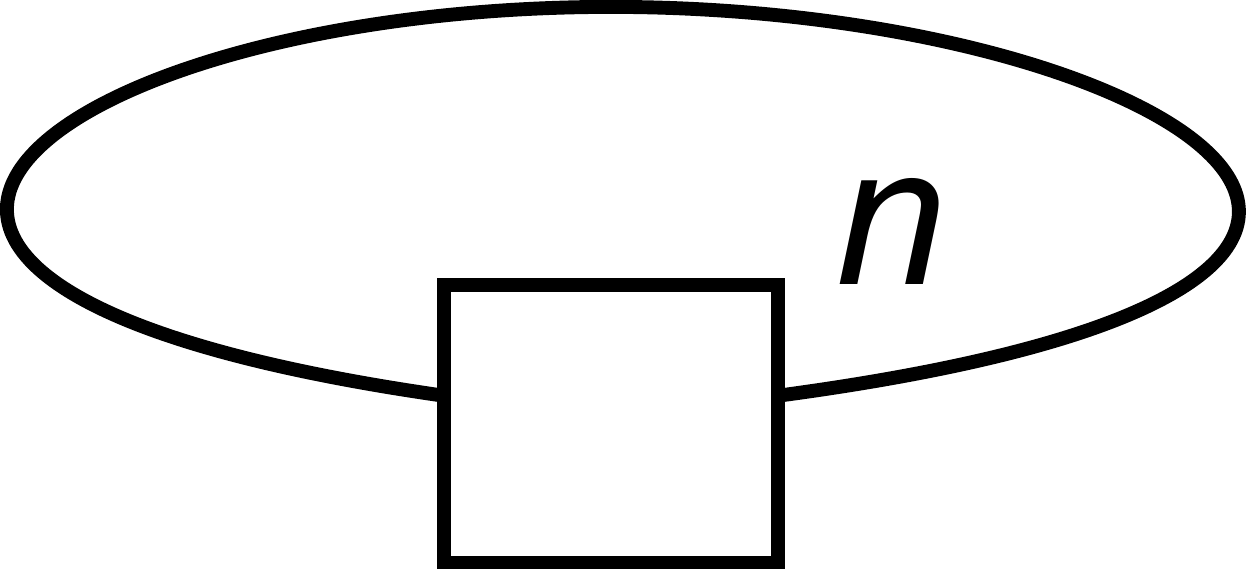}}\right \rangle=(-1)^n[n+1]\end{equation}
where \begin{equation}[n]=\frac{\{n\}}{\{1\}} \textrm{, } \{n\}=A^{2n}-A^{-2n}\textrm{ and } A^{-4}=a^{-2}=q.\end{equation}

  It also relies on the evaluation of $\theta(a,b,c)$. Assume $(a, b, c)$ is an admissible triple, then let $i,j,k$ be the internal colors, in particular 
\begin{equation} i=(b+c-a)/2 \qquad \quad j=(a+c-b)/2 \qquad \quad k=(a+b-c)/2.\end{equation} The condition of being an admissible triple is exactly the condition that makes $i,j$ and $k$ positive integers. 

The term $\theta(a,b,c)$ is the trihedron coefficient. In particular, 
\begin{equation}\theta(a,b,c) = \left \langle\raisebox{-.25
\height} {\includegraphics{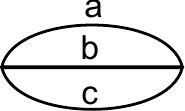}} \right \rangle= (-1)^{i+j+k}\frac{[i+j+k+1]![i]![j]![k]!}{[i+j]![j+k]![i+k]!}.\end{equation}

Once we use fusion, we get rid of twists using $\gamma$, the negative half twist coefficient. 
\begin{equation}\raisebox{-.5\height}{\includegraphics[width=.1\textwidth]{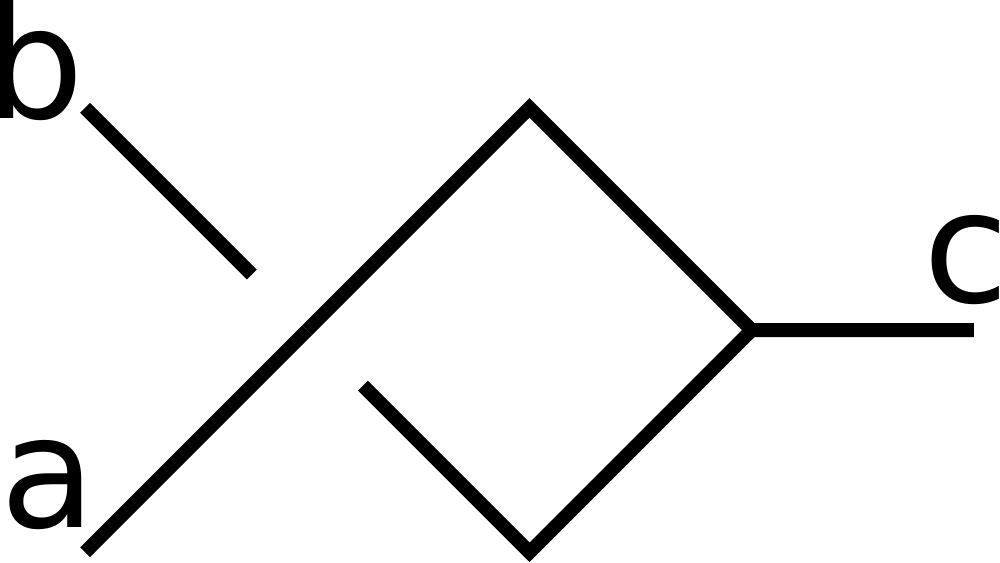}}=\gamma(a,b,c)\raisebox{-.5\height}{\includegraphics[width=.1\textwidth]{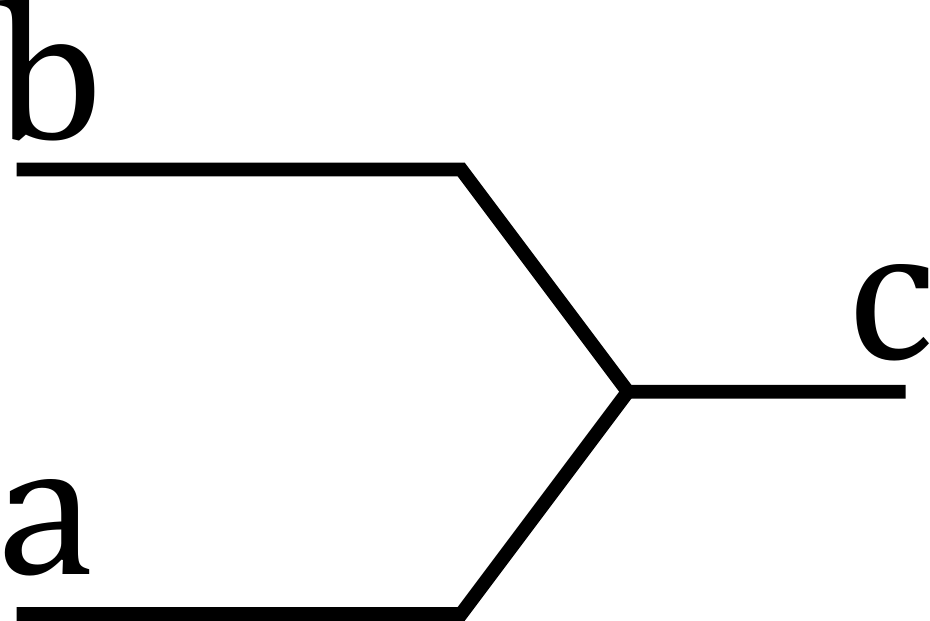}}\end{equation}
with 
 \begin{equation}\gamma(a,b,c )=(-1)^\frac{a+b-c}{2} A^{a+b-c+\frac{a^2+b^2-c^2}{2}}. \end{equation}


To use these to find the colored Jones polynomial of a knot, we identify all of the negative twist regions of the knot diagram $D$.  A negative twist region is a section of the knot with one or more consecutive negative twists. (Here, we think of a local orientation where each strand is oriented in the same direction, as opposed to a global orientation, when defining negative twists.) We do fusion along each of the $k$ negative twist regions and then use the twist coefficients to remove all crossings. 

 Each of the $k$ regions corresponding to the negative twists will be labeled with $2j_i$ for $1\le i \le k$. The other edges will be labeled $n$. Thus we have a multi-sum of trivalent graphs, $\Gamma_{n,(j_1,\ldots ,j_k)}$ where each $j_i$ runs from $1$ to $n$.  The evaluation of this sum gives us the $n+1$ colored Jones polynomial of the knot $K$. i.e.,
\begin{equation}J_{n+1,K}=\sum_{j_1,\ldots j_k=0}^{n}\prod_{i=1}^{k}\gamma(n,n,2j_i)^{m_i}\frac{\Delta_{2j_i}}{\theta(n,n,2j_i)}\Gamma_{n,(j_1,\ldots ,j_k)}\end{equation}

Armond and Dasbach show that when we decrease a single $j_i$ from $n$ to $n-1$, we decrease the highest degree (in $q$) by $n+1$ and we continue to decrease it as we decrease $j_1$ further, so the only graphs that contribute to the highest $n+1$ terms of the colored Jones polynomial, the head, are the ones where all $j_i=n$.

We will use the notation $f(q)\stackrel{\cdot n}{=}g(q)$ if the highest $n$ terms of $f(q)$ agree with those of $g(q)$, i.e once we multiply by some power of $q$ to get $f^*$ and $g^*$ so that $f^*(q)$ and $g^*(q)$ both have highest degree equal to $n$, then $f^*(q)-g^*(q)$ has highest degree $k$ for some $k\leq 0$.

  Thus we have 
\begin{equation}J_{n+1,K}\stackrel{\cdot n+1}{=}\prod_{i=1}^{k}\gamma(n,n,2n)^{m_i}\frac{\Delta_{2n}}{\theta(n,n,2n)}\Gamma_{n,(n,\ldots ,n)}.\end{equation}

When calculating $J_{n+1,K},$ we calculate in terms of the variable $A$ where $A^{-4}=q$, so looking at the highest degree terms of $q$ is the same as looking at the lowest degree terms of $A$, which is where we will do our work. For a rational function $R$, let $d(R)$ be the minimum degree of $R$ considered as a power series when you expand $\mathbb{Z}(q) \hookrightarrow \mathbb{Z}[q^{-1},q]]$.  We choose to expand the power series allowing for infinite terms in the positive direction so that we get a minimum degree. 

The following are the three lemmas proved in \cite{DasArm}.  We will use these lemmas to prove an extension on their work. 

\begin{lemma}[\cite{DasArm}]
When we decrease $j_i$ from $n$ to $n-1$, we increase the minimum  $A$ degree of $\gamma(n,n,j_i)$ by $4n$, i.e.,
\begin{equation}d(\gamma(n,n,2(n-1)))=d(\gamma(n,n,2n))+4n.\end{equation}
As we continue to decrease each $j_i$, the minimum $A$ degree continues to increase, i.e.,
\begin{equation}d(\gamma(n,n,2(j-1))) \geq d(\gamma(n,n,2j)).\end{equation}
\label{gammalemma}
\end{lemma}

\begin{lemma}[\cite{DasArm}]
Each time we decrease the $j_i$, we increase the minimum $A$ degree of $\frac{\Delta_{2j}}{\theta(n,n,2j)}$ by 2, i.e.,
\begin{equation}d\left(\frac{\Delta_{2(j-1)}}{\theta(n,n,2(j-1))}\right)=d\left(\frac{\Delta_{2j}}{\theta(n,n,2j)}\right)+2.\end{equation}
\end{lemma}

\begin{lemma}[\cite{DasArm}]
When we decrease $j_i$ from $n$ to $n-1$, we increase the minimum $A$ degree of $\Gamma_{n,(j_1,\ldots ,j_{i-1},j_i,j_{i+1},\ldots,j_k)}$ by at least 2, i,e.,
\begin{equation}d(\Gamma_{n,(n,\ldots ,n-1,\ldots,n)})\geq d(\Gamma_{n,(n,\ldots ,n,\ldots,n)})+ 2.\end{equation}
We can only guarantee this change of two at the first step.  As we continue to decrease $j_i$, the best we get is:
\begin{equation}d(\Gamma_{n,(j_1,\ldots ,j_{i-1},j_i-1,j_{i+1},\ldots,j_k)})\geq d(\Gamma_{n,(j_1,\ldots ,j_{i-1},j_i,j_{i+1},\ldots,j_k)})\pm 2.\end{equation}

\end{lemma}
Thus when we decrease a single $j_i$ from $n$ to $n-1$, we increase the lowest degree in $A$ by at least $(4n)m_i+2+2\geq 4n+4$ which means for the $n+1$ Colored Jones Polynomial, we decrease the $q$ degree by $n+1$ .  After this, the lowest degree in $A$ is non-decreasing as we continue to change the $j_i$ so the only term that contributes to the lowest $4n+4$ terms of the polynomial in $A$ are the ones where each $j_i=n$.

One can then show that the first $n+1$ coefficients only depend on the overall graph structure and not the number of twists.  This is the same information we lose when going to the reduced graph. This gives the proof of Theorem \ref{tailstab}.  It does not, however, tells us anything about the existence of the head and tail of alternating or adequate knots. It just tells us that if the head and tail exist, they only depend on the overall graph structure.

To show the head (tail) exists, we must show that the highest (lowest) $n+1$ coefficients of the evaluation of the $n$-trivalent graph is the same as those of the evaluation of the $n+1$-trivalent graph. Armond does this in \cite{Arm} by demonstrating a way to reduce the $n+1$-colored graph to the $n$-colored graph by peeling off one of the strands without changing the highest $n+1$ coefficients.

We can use similar techniques to those used by Armond and Dasbach to find which knots will have the same higher order stabilizing sequences.
 
 First, we state a direct corollary to the theorem above in the case where each of the twist coefficients is large. 

\begin{corollary} 
Let $m$ be the minimum of the $m_i$. When we change a single $j_i$ from $n$ to $n-1$, we increase the lowest degree in $A$ by at least $(4n)m+2+2 = 4nm+4$ and thus decrease the $q$ degree by $n(m)+1$ for the $(n+1)$ Colored Jones Polynomial.  In addition to the first $n+1$ terms only depending on the overall graph structure, the next $(m-1)n$ terms also depend only on the graph structure.
\end{corollary}


There is also a stronger theorem, which was stated in the introduction.  It deals with the case where some of the twist coefficients are large but others are small. We restate the theorem here and then proceed with the proof.  \\

 \noindent \textbf{Theorem 1.1.} \emph{If $K_1$ and $K_2$ are alternating knots whose alternating diagrams have the same $m+1$-reduced B-checkerboard graph structure, then the highest (m+1)N coefficients (in $q$) of $J_{N,K_1}$ are the same as the first $(m+1)N$ coefficients of $J_{N,K_2}$, up to possible rescaling by $\pm 1$.}\\

\begin{proof}
Let $b$ be the number of pairs of vertices in the B-checkerboard graph that have more than $m$ parallel edges. Relabel the graph such that $j_1, \cdots, j_b$ correspond to these high order parallel edges. By Lemma \ref{gammalemma}, when any of these $j_i$ decrease from $N$ to $N-1$, there is an overall decrease in the maximum degree of $(m+1)(N)$. Since the other terms in the sum do not increase the maximum degree, we also know that the degrees continue to decrease as we decrease the value of any $j_i$. Thus to contribute to the highest $(m+1)N$ coefficients, and thus contribute to $\Phi_m(q)$, each of $j_1,\dots,j_b$ needs to be labeled $N$.  This means we can write the colored Jones polynomial as   
\begin{align}
J_{N+1,K}&\stackrel{\phantom{\cdot (m+1)N}}{=}\sum_{j_1,\ldots j_k=0}^{N}\prod_{i=1}^{k}\gamma(N,N,2j_i)^{m_i}\frac{\Delta_{2j_i}}{\theta(N,N,2j_i)}\Gamma_{N,(j_1,\ldots ,j_k)}& \nonumber \\
&\stackrel{\cdot (m+1)N}{=} \gamma(N,N,2N)^{\sum_{i=1}^b m_i} & \nonumber \\ 
&\qquad \times \sum_{j_{b+1},\ldots, j_k=0}^{N}\prod_{i=b+1}^{k}\gamma(N,N,2j_i)^{m_i}\prod_{i=1}^{k}\frac{\Delta_{2j_i}}{\theta(N,N,2j_i)}\Gamma_{N,(N,\ldots,N,j_{b+1},\ldots ,j_k)}&
\end{align}
Again, since $\gamma(N,N,2N)$ only contributes an overall shift, we get the same highest $(m+1)N$ coefficients regardless of the values of $m_1, \dots, m_b$ and thus knots with the same $m+1$-reduced graph structure have the same highest $(m+1)N$ coefficients. \\

\end{proof}

\section{Knots which reduce to a triangle graph}
\label{trianglegraph}
 \begin{figure}[ht]
 \centering
 \includegraphics[width=.25\textwidth]{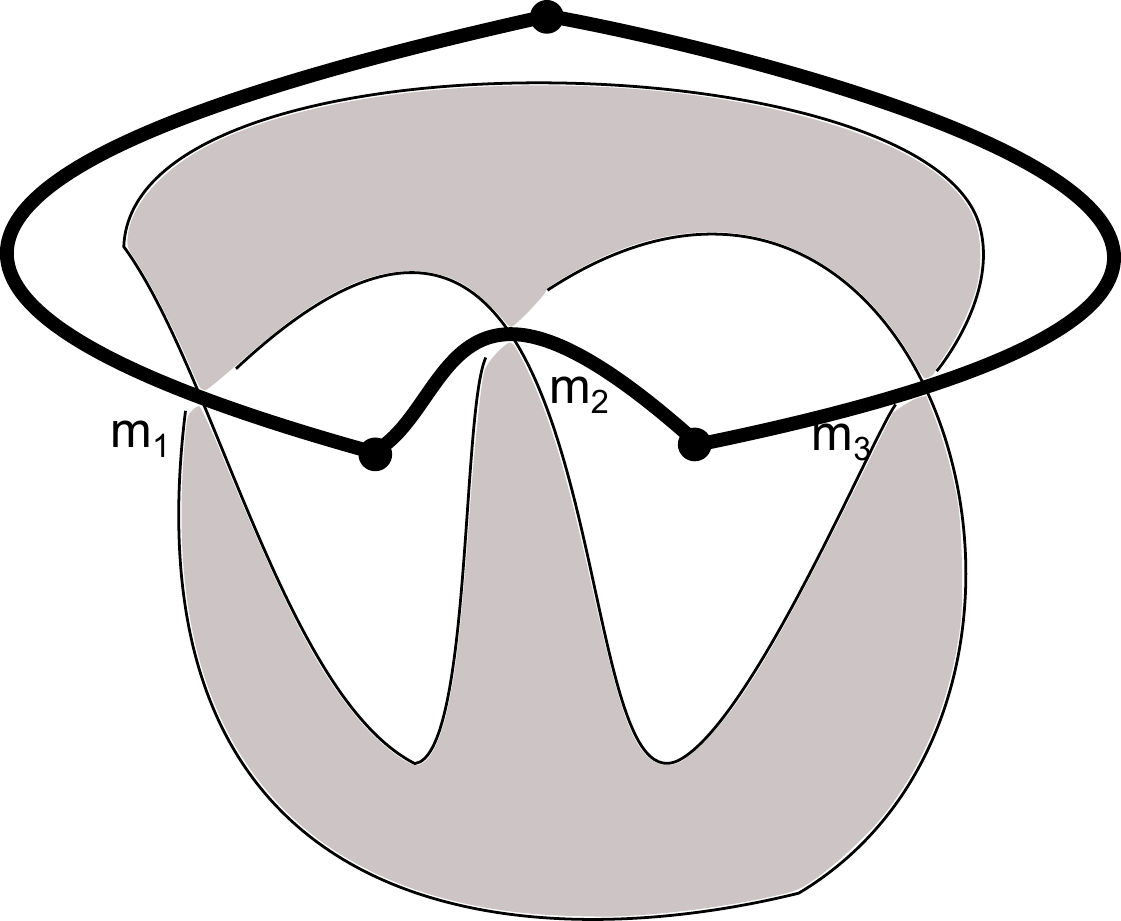}
  \caption{A trefoil knot with its checkerboard graph.}
 \label{treftail}
 \end{figure}
In this section, we will find $T_1$, the ``neck" of the colored Jones polynomial of knots whose reduced checkerboard graph is a triangle graph. We define $T_1$ to be the polynomial which when added in the right way to $T_0$, the head, will have the property that its highest order $2N+1$ terms agree with the highest order $2N+1$ terms of the $N+1$ colored Jones polynomial.   The knots we will focus on can be drawn like the trefoil in Fig. \ref{treftail}, except we will have more crossings below the pictured crossings  (and thus more parallel edges before we reduce the graph).  The $m_i$ represent the number of crossings in each section.  As it is drawn, each $m_i=1$.  If $m_1=2$ and the others are 1, we get the figure 8 knot.

As before, we find the colored Jones polynomial by doing fusion and removing negative twists. 

\begin{align}
J_{N+1,K}(q)	&=\left\langle \raisebox{-.5\height}{\includegraphics[width=.15\textwidth]{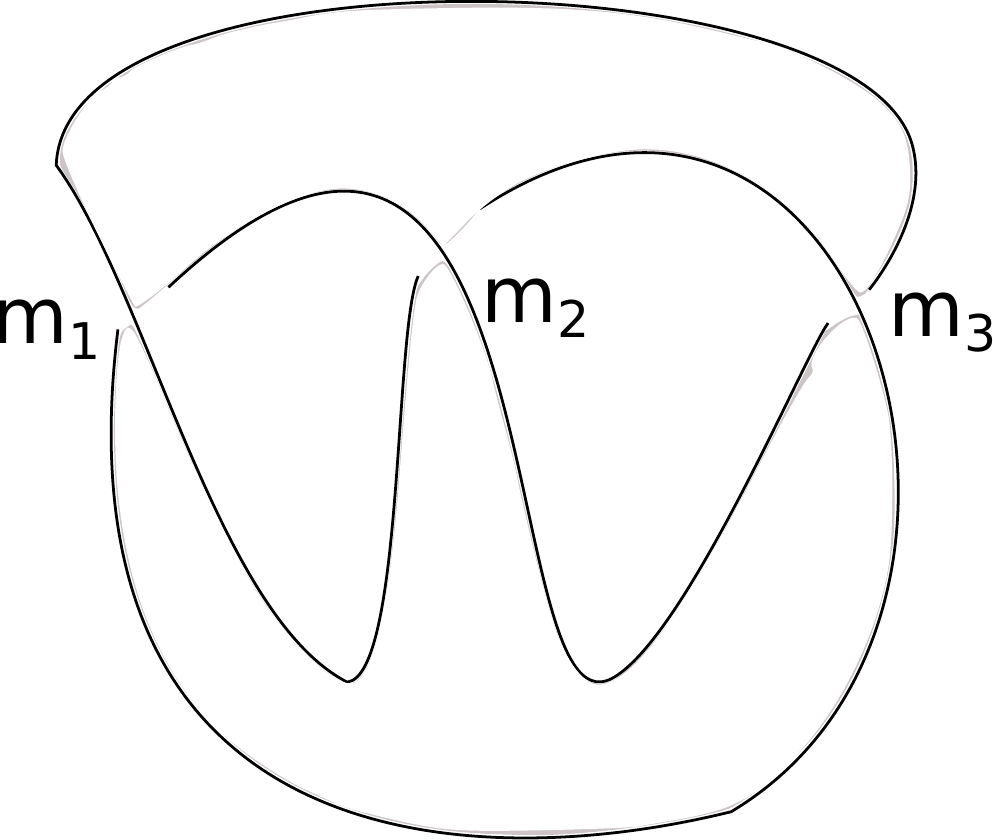}}\right\rangle \nonumber \\
				&=\sum_{j_i=0}^N\prod_{i=1}^3 \frac{\Delta_{2j_i}}{\theta(N,N,2j_i)}\left\langle \raisebox{-.5\height}{\includegraphics[width=.15\textwidth]{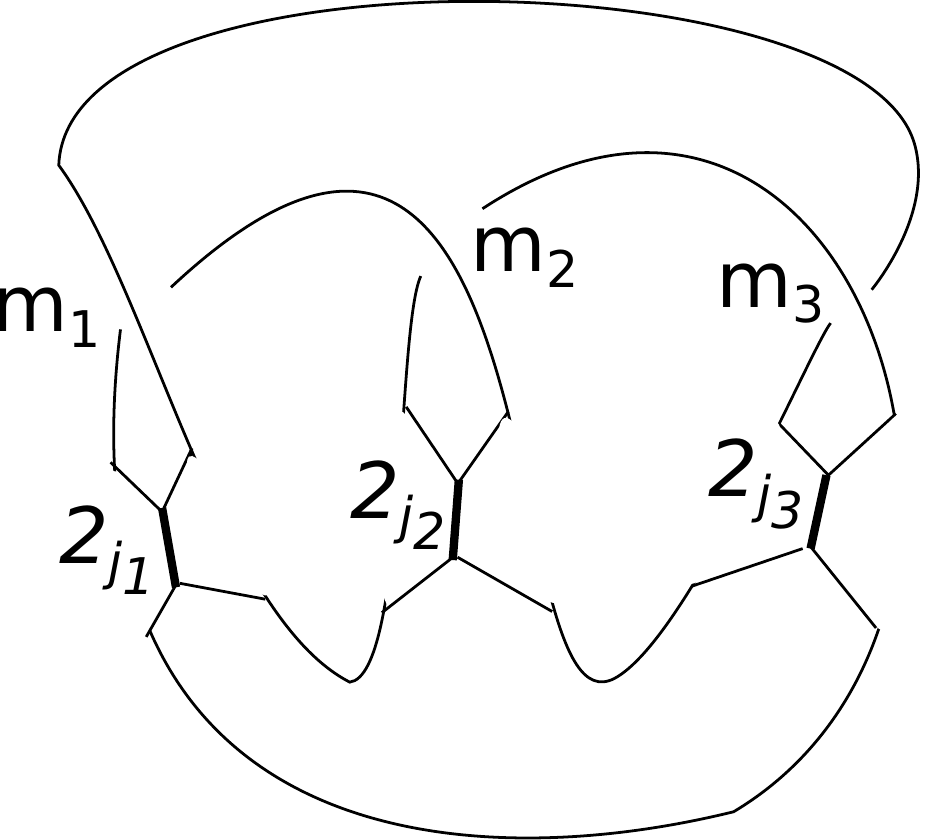}}\right\rangle \nonumber \\
				&=\sum_{j_i=0}^N\prod_{i=1}^3 \gamma(N,N,2j_i)^{m_i}\frac{\Delta_{2j_i}}{\theta(N,N,2j_i)}\left\langle \raisebox{-.5\height}{\includegraphics[width=.15\textwidth]{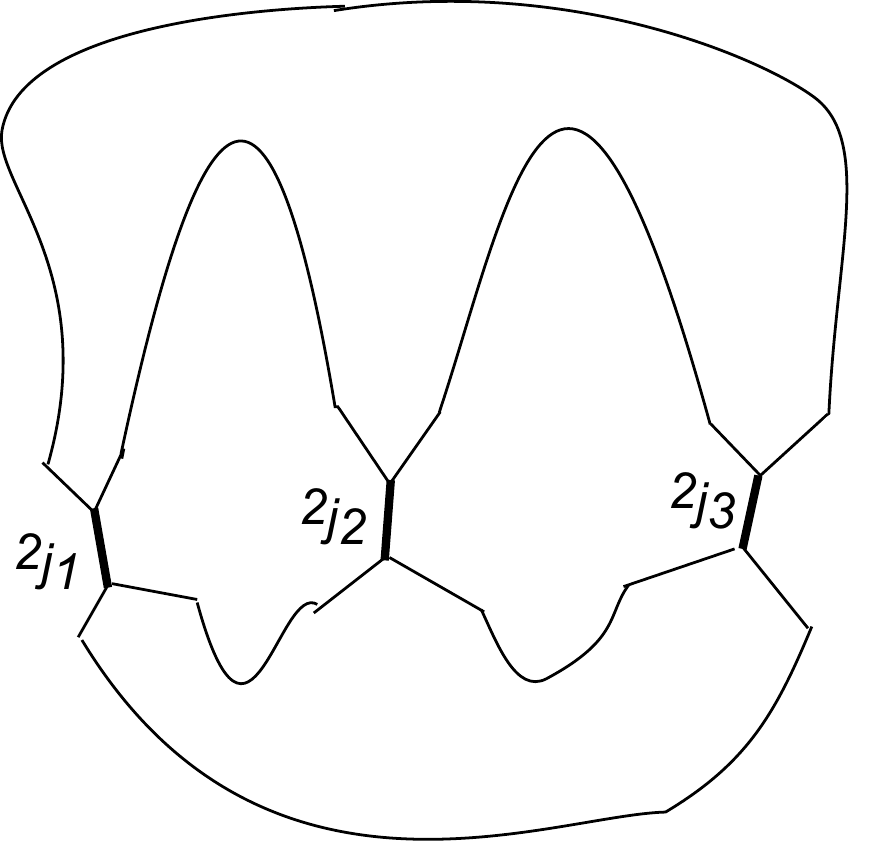}}\right\rangle \nonumber \\
				&=\sum_{j_i=0}^N\prod_{i=1}^3 \gamma(N,N,2j_i)^{m_i}\frac{\Delta_{2j_i}}{\theta(N,N,2j_i)}\Gamma_{N,(j_1,j_2,j_3)}
\end{align}
where
\begin{equation}\Gamma_{N,(j_1,j_2,j_3)}=\left\langle \raisebox{-.5\height}{\includegraphics[width=.2\textwidth]{treftailneck3.pdf}}\right\rangle=\left\langle \raisebox{-.5\height}{\includegraphics[width=.2\textwidth]{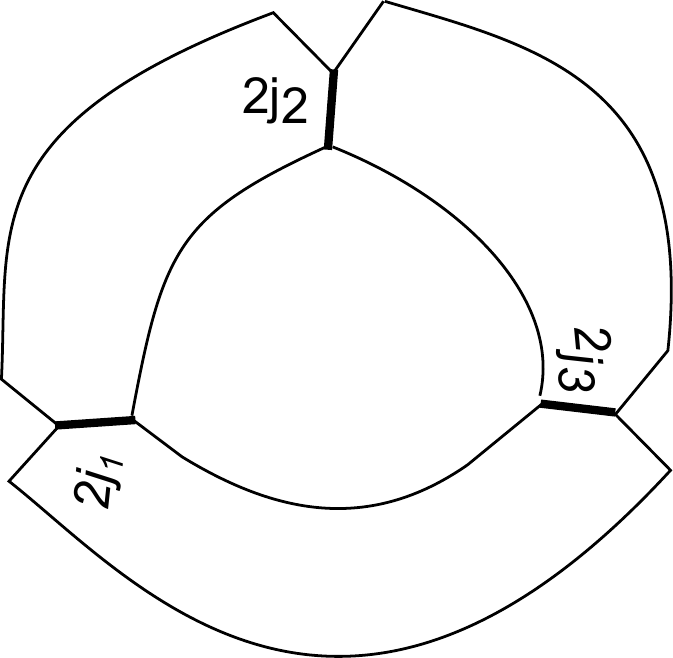}}\right\rangle \end{equation}
Now, compare this diagram to the diagram in Fig. \ref{gammaxyz}. Following \cite{Lic}, we denote the evaluation of the graph in Fig. \ref{gammaxyz} as $\Gamma(x,y,z)$.
\begin{figure}
\centering 
\includegraphics[width=.3\textwidth]{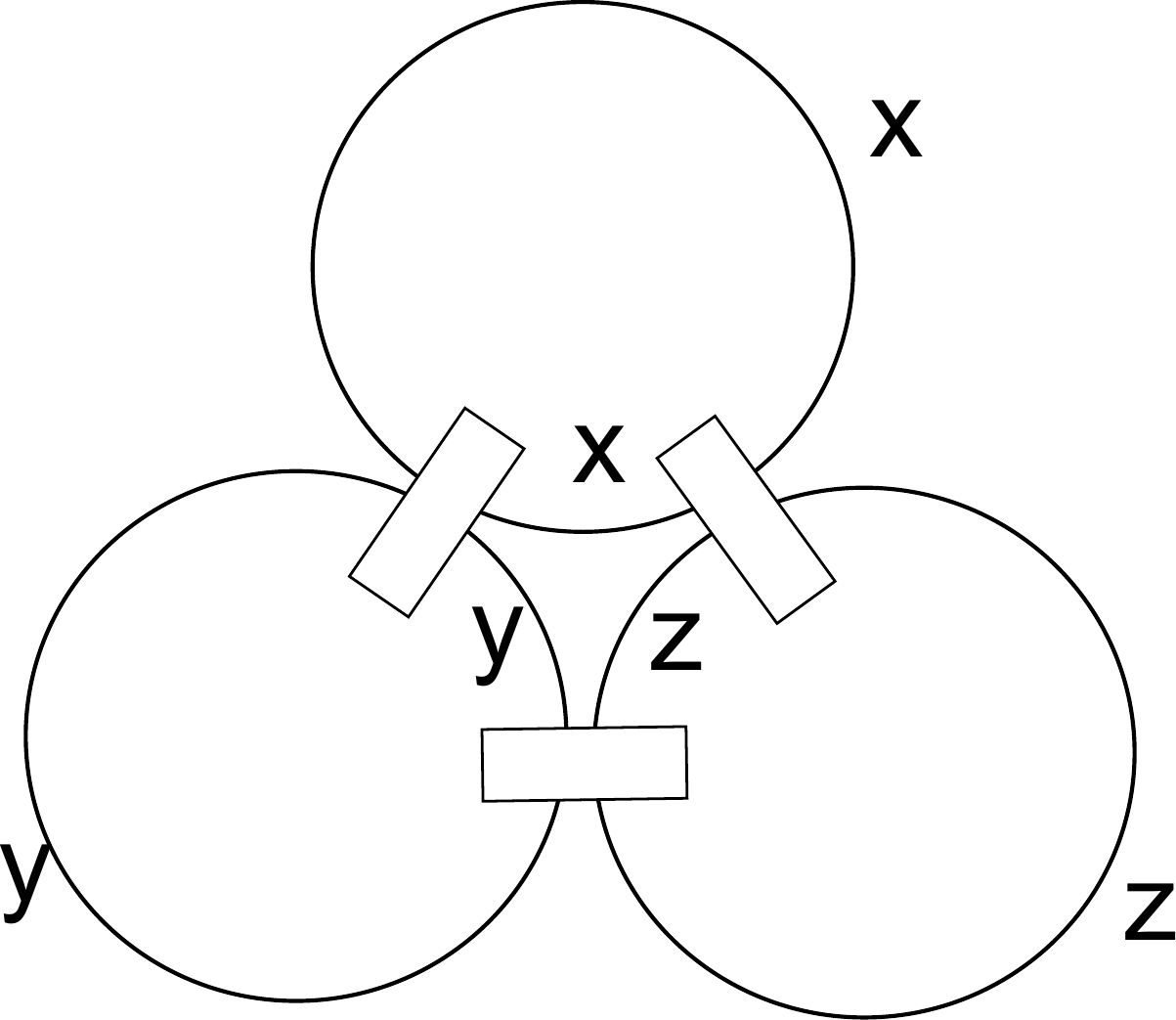}
 \caption{This diagram has $x$ parallel copies of one circle, $y$ of another and $z$ of a third. They are joined by the $x+y, y+z,$ and $x+z$ idempotents. Its evaluation is $\Gamma(x,y,z)$}
\label{gammaxyz}
\end{figure}

We are interested in finding the highest order $2N+1$ coefficients (in the variable $q$) of the $N+1$ colored Jones polynomial. To find the highest order coefficients of the polynomial in $q$, we need to find the lowest order in $A$ since $q=A^{-4}$ First, we prove we only need to consider the cases where either all $j_i=N$ or where exactly one $j_i=N-1$. Then, we will look at $\Gamma_{N,(j_1,j_2,j_3)}$ in these two cases.

Consider $J_{N+1,K}(q)$. We already saw that the lowest order terms in $A$, and thus highest order terms in $q$, come from the summand where all $j_i$ are $N$. We also saw that when we decrease a single $j_1$ from $N$ to $N-1$, we decrease the maximum degree (in $q$) by $(N+1)$. And hence the $(N,N,N)$ labeled term is the only one that contributes to the head. Now, we want to find the next $N$ terms. We know that the terms with one $N-1$ and two $N$ labelings count toward these terms, but we want to know if any other terms contribute.  Decreasing a different $N$ to $N-1$ will again decrease the maximum degree by $N+1$ and thus not contribute to the next $N$ terms.  

If we decrease the $N-1$ label to $N-2$, we decrease the maximum degree coming from $\gamma(N,N,2j)$ by $N-1$. In general, when decreasing a general $j$ to $j-1$ the $\frac{\Delta_{2j_i}}{\Theta(N,N,2j_i)}$ term and $\Gamma_{N,(j_1,j_2,j_3)}$ term may cancel each other's changes since $\Gamma_{N,(j_1,j_2,j_3)}$ changes by $\pm 2$. However, by comparing the specific $\Gamma_{N,(N,N,N-1)}$ and $\Gamma_{N,(N,N,N-2)}$ in this case, we can see that again we are decreasing the number of circles and thus these together decrease the maximum degree (in $q$) by 1.  We can conclude that the $(N,N,N-2)$ terms also does not contribute to highest $2N+1$ coefficients.   Since the maximum degrees (in $q$) continues to decrease as we decrease the labels, the $(N,N,N)$ and $(N-1,N,N)$ terms (up to permutation) are the only ones we need to consider.

Now, we find the evaluation of the graph in each of these cases. In the case where each $j_i$ is N, it is easy to see that \begin{equation}\Gamma_{N,(N,N,N)}=\Gamma(N,N,N).\end{equation}

In the case where one is $N-1$, we can expand the fusion into the idempotents, see Fig. \ref{treftailneckjisN1}.
\begin{figure}[h]
\centering
\includegraphics[width=.3\textwidth]{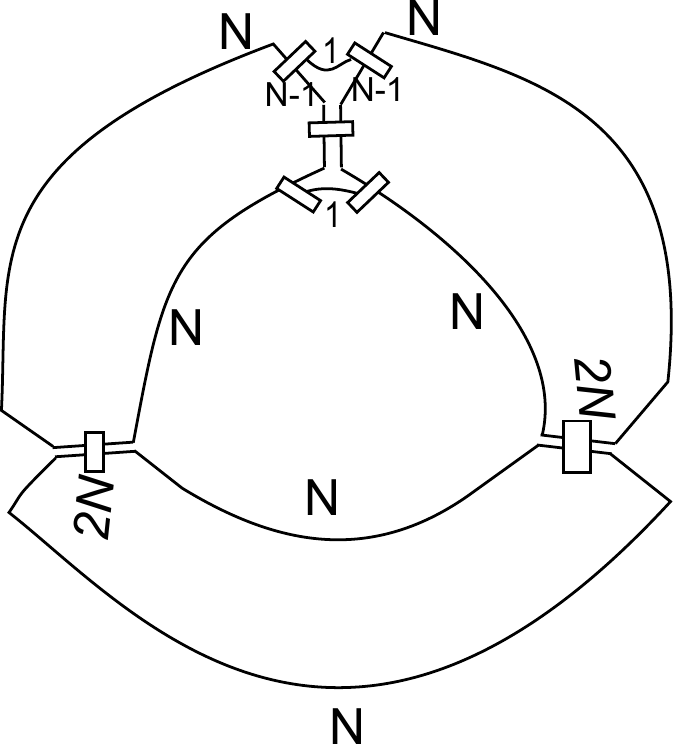}
 \caption{This diagram shows the expansion of the fusion with $j_i=N-1$ into the idempotent form.}
\label{treftailneckjisN1}
\end{figure}
The $N$ idempotents can be absorbed into the $2N$ idempotents, see Fig. \ref{treftailneckjisN12}.  We can then pull the outer and inner single strand down.  Doing this we can see that \begin{equation}\Gamma_{N,(N-1,N,N)}=\Gamma(N+1,N-1,N-1)).\end{equation}
\begin{figure}[h]
\centering
\includegraphics[width=.3\textwidth]{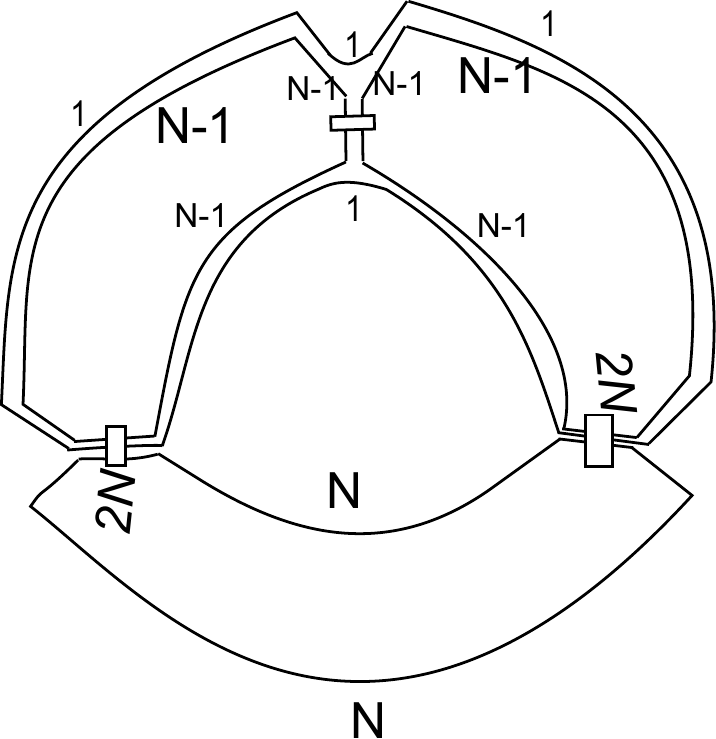}
 \caption{We can get rid of the $N$ idempotents.  Then after moving the 1 strands, we can see that $\Gamma_{N,(N-1,N,N)}=\Gamma(N+1,N-1,N-1).$ }
\label{treftailneckjisN12}
\end{figure}

 \subsection{Finding an expression for $T_1$}
 
 Recall, if two polynomials, $f(q)$ and $g(q)$ have the same coefficients for the $n$ highest order terms, we will write \begin{equation} f(q)\stackrel{\cdot n}{=}g(q). \end{equation} This is also used to mean that the lowest degree terms in $A$ match for $4n$ terms. For notational convenience, if $f(q)$ is a Laurent polynomial (or has a power series representation as a Laurent polynomial whose $q$ terms are bounded), let $f_{n}(q)$ represent a polynomial whose highest $n$ terms agree with $f(q).$

\begin{lemma}
If $f(q)=\frac{g(q)}{h(q)}$ then $f_n(q)\stackrel{\cdot n}{=}\frac{g_{n}(q)}{h_{n}(q)}$.
\label{divlemma}
\end{lemma}
\begin{proof} [Proof of \ref{divlemma}]

If $f(q)=\frac{g(q)}{h(q)}$ then $f(q)h(q)=g(q)$. Since the highest $n$ terms of $f(q)$ and $g(q)$ determine the highest $n$ terms of $h(q)$, $f_{n}(q)h_{n}(q)=g_{n}(q)$ so   $f_{n}(q)=\frac{g_{n}(q)}{h_{n}(q)}$.
\end{proof}

In order to find $T_1$, we will reduce our formula for the $N+1$ colored Jones polynomial as much as possible while keeping the highest $2N+1$ coefficients the same. To help keep our calculations as clear as possible, we will first prove a few lemmas.

\begin{lemma}
\begin{equation}\{2N\}! \stackrel{\cdot 2N+1}{=}(-1)^Nq^{\frac{3N^2+N}{4}} \{N\}!\left(1-\frac{q^{-N-1}}{1-q^{-1}}\right) \end{equation}
\label{2nlemma}
\end{lemma}

\begin{proof}
 We expand the higher terms in the factorial.  In the

\noindent $\left(q^{-(N+i)/2} - q^{(N+i)/2}\right)$ terms changing from the $-q^{(N+i)/2}$ term to the $q^{-(N+i)/2}$ term decreases the degree by $N+i$. Thus, we can either not do this at all or only do this once. We refer to the $q^{-(N+i)/2}$ term as a $q^-$ term. 

\begin{align}
\{2N\}!&=\{2N\}\{2N-1\}\cdots\{N+1\}\{N\}! \nonumber \\
				&= (q^{-(N+N)/2}-q^{(N+N)/2})\cdots (q^{-(N+1)/2}-q^{(N+1)/2})\{N\}! \nonumber \\
				&\stackrel{\cdot 2N+1} {=}  (-1)^N q^{\sum_{i=1}^{N}(N+i)/2} \left(\underbrace{\{N\}!}_{\text{no $q^-$ terms}} - \underbrace{\sum_{i=1}^{N} q^{-(N+i)}\{N\}!}_{\text{one $q^-$ term}}\right) \nonumber \\
				&= (-1)^Nq^{\frac{3N^2+N}{4}}\{N\}!\left(1 -\sum_{i=1}^{N} q^{-N-i}\right) \nonumber \\
				&= (-1)^Nq^{\frac{3N^2+N}{4}}\{N\}!\left(1 -\frac{q^{-N-1}-q^{-2N-1}}{1-q^{-1}}\right) \nonumber \\
				&= (-1)^Nq^{\frac{3N^2+N}{4}}\{N\}!\left(1 - \frac{q^{-N-1}}{1-q^{-1}}\underbrace{+\frac{q^{-2N-1}}{1-q^{-1}}}_{\text{no effect on highest $2N+1$ terms}}\right) \nonumber \\
				&\stackrel{\cdot 2N+1} {=}  (-1)^Nq^{\frac{3N^2+N}{4}}\{N\}!\left(1 -\frac{q^{-N-1}}{1-q^{-1}}\right)
\end{align}
 This concludes the proof.

\end{proof}

Now using this lemma, we can get an expression for the first $2N+1$ terms of $\{2N\}!^2$.
\begin{corollary}
\begin{equation}\left(\{2N\}!\right)^2 \stackrel{\cdot 2N+1}{=} q^{\frac{3N^2+N}{2}} \{N\}!^2\left(1-\frac{2q^{-N-1}}{1-q^{-1}}\right) \end{equation}
\label{2n2cor}
\end{corollary}

\begin{proof}
\begin{eqnarray}
\left(\{2N\}!\right)^2 	& \stackrel{\cdot 2N+1}{=}& \left((-1)^Nq^{\frac{3N^2+N}{4}}\{N\}!\left(1 -\frac{q^{-N-1}}{1-q^{-1}}\right)\right)^2 \nonumber \\
				& =&q^{\frac{3N^2+N}{2}} \{N\}!^2 \left(1-\frac{2q^{-N-1}}{1-q^{-1}}+\frac{q^{-2N-2}}{(1-q^{-1})^2}\right) \nonumber \\
				& =&q^{\frac{3N^2+N}{2}}\{N\}!^2 \left(1-\frac{2q^{-N-1}}{1-q^{-1}}+\frac{q^{-2N-2}}{1-(2q^{-1}-q^{-2})}\right) \nonumber \\
				& =&q^{\frac{3N^2+N}{2}} \{N\}!^2 \left(1-\frac{2q^{-N-1}}{1-q^{-1}}+\underbrace{q^{-2N-2}(1+(2q^{-1}-q^{-2})+\cdots)}_{\text{no effect on first $2N+1$ terms}}\right) \nonumber \\	
				& \stackrel{\cdot 2N+1}{=} &q^{\frac{3N^2+N}{2}} \{N\}!^2 \left(1-\frac{2q^{-N-1}}{1-q^{-1}}\right).			
\end{eqnarray}
This concludes the proof.
\end{proof}
\subsection{Knots with $m_i>2$}

In the case where each $m_i$ is greater than 2, the maximum degree decreases by more than $2N$ when we decrease $j_i$ from $N$ to $N-1$, thus we only need to deal with the case where each $j_i=N$. Thus we get
\begin{eqnarray}
J_{N+1,K}(q)&=&\sum_{j_i=0}^N\prod_{i=1}^3 \gamma(N,N,2j_i)^{m_i}\frac{\Delta_{2j_i}}{\theta(N,N,2j_i)}\Gamma_{N,(j_1,j_2,j_3)}\nonumber \\
		&\stackrel{\cdot 2N+1}{=}&\prod_{i=1}^3 \gamma(N,N,2N)^{m_i}\frac{\Delta_{2N}}{\theta(N,N,2N)}\Gamma_{N,(N, N, N)}\nonumber \\
		&=&\gamma(N,N,2N)^{m_1+m_2+m_3}\left(\frac{\Delta_{2N}}{\theta(N,N,2N)}\right)^3\left(\Gamma_{N, (N, N, N)}\right).
\end{eqnarray}
Recall that $\gamma(a, b, c)=(-1)^{\frac{a+b-c}{2}}A^{a+b+c+\frac{a^2+b^2-c^2}{2}}$.  This just has the effect of shifting polynomial but does not affect the sequence of coefficients. Also
\begin{equation}\Delta_n=\frac{(-1)^n(a^{n+1}-a^{-(n+1)})}{a-a^{-1}}=\frac{(-1)^n\{n+1\}}{\{1\}}
\end{equation}
and 

\begin{equation}\Delta_n!=\Delta_n \Delta_{n-1}\cdots \Delta_1 = (-1)^{\frac{n(n+1)}{2}}\frac{\{n+1\}!}{\{1\}^{n+1}}.\end{equation}

Following Lickorish \cite{Lic}, we define $\Gamma(x,y,z)$ to be the diagram consisting of $x$ parallel copies of a circle, $y$ parallel copies of a circle and $z$ parallel copies of a circle joined by the $x+y, y+z,$ and $z+x$ idempotents. This is what our knot reduces to, i.e. $\Gamma_{N,(N,N,N)}=\Gamma(N,N,N)$.
\begin{lemma} {\cite{Lic}}

\begin{equation} \Gamma(x,y,z)=\frac{\Delta_{x+y+z}!\Delta_{x-1}!\Delta_{y-1}!\Delta_{z-1}!}{\Delta_{y+z-1}!\Delta_{z+x-1}!\Delta_{x+y-1}!}. 
\end{equation}
\label{Gammalem}
\end{lemma}
Also

\begin{equation} \theta(N,N,2N)=\Gamma(N,N,0)=\frac{\Delta_{2N}!\Delta_{N-1}!\Delta_{N-1}!}{\Delta_{N-1}!\Delta_{N-1}!\Delta_{2N-1}!}=\Delta_{2N}.
\end{equation}
So $\left(\frac{\Delta_{2N}}{\theta(N,N,2N)}\right)=1$, and we have:
\begin{eqnarray}
J_{N+1,K}(q)	&\stackrel{\cdot 2N+1}{=}&\gamma(N,N,2N)^{m_1+m_2+m_3}\left(\frac{\Delta_{2N}}{\theta(N,N,2N)}\right)^3\left(\Gamma_{N,(N,N,N)}\right) \nonumber \\
			&\stackrel{\cdot 2N+1}{=}&\Gamma(N,N,N) \nonumber \\
			&=&\frac{\Delta_{3N}!\Delta_{N-1}!\Delta_{N-1}!\Delta_{N-1}!}{\Delta_{2N-1}!\Delta_{2N-1}!\Delta_{2N-1}!} \nonumber \\
			&=&\frac{\Delta_{3N}!\Delta_{N-1}!^3}{\Delta_{2N-1}!^3}  \nonumber \\
			&=&\frac{(-1)^{N}\{3N+1\}!\{N\}!^3}{\{2N\}!^3\{1\}}.
\end{eqnarray}
Again, since we only want the  $2N+1$ terms with highest $q$ degree, we can reduce the $\{3N+1\}!$ term.  We know we have to choose the $-q^{(2N+i)}$ term for each $1\le i \le N+1$. This gives us a shift in degree that we can ignore. It also gives a factor of $(-1)^{N+1}$. We have
\begin{eqnarray}
J_{N+1,K}(q)	&\stackrel{\cdot 2N+1}{=}&\frac{(-1)^{N}\{3N+1\}!\{N\}!^3}{\{2N\}!^3\{1\}} \nonumber \\
			&\stackrel{\cdot 2N+1}{=}&\frac{(-1)^{1}\{2N\}!\{N\}!^3}{\{2N\}!^3\{1\}} \nonumber \\
			&=&\frac{(-1)\{N\}!^3}{\{2N\}!^2\{1\}}.
\end{eqnarray}
By Lemma \ref{divlemma}, we know that if $f(q) \stackrel{\cdot k}{=}g(q)$, then $1/f(q) \stackrel{\cdot k}{=}1/g(q)$. We can reduce the $\{2N\}!^2$ term to its lowest $2N+1$ terms, which we found in Corollary \ref{2n2cor}, again forgetting about the shifting. Thus we have:
\begin{eqnarray}
J_{N+1,K}(q)	&\stackrel{\cdot 2N+1}{=}&\frac{(-1)\{N\}!^3}{\{2N\}!^2\{1\}} \nonumber\\
			&\stackrel{\cdot 2N+1}{=}&\frac{(-1)\{N\}!^3}{\{N\}!^2\left(1-\frac{2q^{-N-1}}{1-q^{-1}}\right)\{1\}} \nonumber \\
			&=&\frac{(-1)\{N\}!}{\left(1-\frac{2q^{-N-1}}{1-q^{-1}}\right)\{1\}}.
\end{eqnarray}
We ultimately want the tailneck of the normalized Colored Jones polynomial, so we will divide by $\Delta_N=\frac{(-1)^{N}\{N+1\}}{\{1\}}$ at this point.

\begin{eqnarray}
J'_{N+1,K}(q)	&\stackrel{\cdot 2N+1}{=}&\frac{(-1)\{N\}!}{\left(1-\frac{2q^{-N-1}}{1-q^{-1}}\right)\{1\}}\frac{\{1\}}{(-1)^N\{N+1\}}\nonumber \\
				&=&\frac{(-1)^{N-1}\{N\}!}{\left(1-\frac{2q^{-N+-1}}{1-q^{-1}}\right)\{N+1\}}\nonumber \\
				&=&\frac{(-1)^{N-1}\{N\}!}{-q^{(N+1)/2}\left(1-\frac{2q^{-N-1}}{1-q^{-1}}\right)\left(1-q^{-N-1}\right)}\nonumber \\
				&\stackrel{\cdot \infty}{=}&\frac{(-1)^N\{N\}!}{1-\frac{2q^{-N-1}}{1-q^{1}}-q^{-N-1}+\underbrace{\frac{2q^{-2N-2}}{1-q^{-1}}}_{\text{does not contribute}}}.
\end{eqnarray}
Now we do a power series expansion of the denominator. Since we only want the highest $2N+1$ terms, we can ignore most of the terms in the expansion.

\begin{eqnarray}
J'_{N+1,K}(q)	&\stackrel{\cdot 2N+1}{=}&\frac{(-1)^N\{N\}!}{1-\left(\frac{2q^{-N-1}}{1-q^{-1}}+q^{-N-1}\right)} \nonumber \\
				&=&(-1)^N\{N\}!\Bigg(1+\left(\frac{2q^{-N-1}}{1-q^{-1}}+q^{-N-1}\right) \nonumber \\
				&&+\underbrace{\left(\frac{2q^{-N-1}}{1-q^{-1}}+q^{-N-1}\right)^2+\cdots}_{\text{do not contribute to highest $2N+1$ terms}}\Bigg) \nonumber\\
				&\stackrel{\cdot 2N+1}{=}&(-1)^N\{N\}!\left(1+\frac{2q^{-N-1}}{1-q^{-1}}+q^{-N-1}\right).
\end{eqnarray}
Note: The (maximum) degree of this term is $\sum_{i=1}^{N}\frac{i}{2}=\frac{N^2+N}{4}$ and its coefficient is $1$.

Now we need to subtract off the stabilized head.  Since the reduced graph is a triangle graph, any knot in this family will have the same head as the figure 8 knot. We need the stabilized head so we take the head of the $2N+1$ colored Jones polynomial of $4_1$, which is $\{2N\}!$. By Lemma \ref{2nlemma} we get

\begin{eqnarray}
\text{stabilized head}	&\stackrel{\cdot 2N+1}{=}&\{2N\}!  \nonumber \\
			&\stackrel{\cdot 2N+1}{=}&(-1)^N \{N\}!\left(1-\frac{q^{-N-1}}{1-q^{-1}}\right).
\end{eqnarray}
Note: The (maximum) degree of this term is $\sum_{i=1}^{N}\frac{i}{2}=\frac{N^2+N}{4}$ and its coefficient is $1$. Thus the maximum degree and sign of the highest $2N+1$ coefficients we found above match so we are set to subtract.
\begin{eqnarray}
J'_{N+1,K}(q)-\text{stabilized head}	&\stackrel{\cdot 2N+1}{=}&(-1)^N\{N\}!\Bigg(\left(1+\frac{2q^{-N-1}}{1-q^{-1}}+q^{-N-1}\right) \nonumber\\
&& - \left(1-\frac{q^{-N-1}}{1-q^{-1}}\right)\Bigg) \nonumber \\
								&=&(-1)^N\{N\}!\left(q^{-N-1}+\frac{3q^{-N-1}}{1-q^{-1}}\right) \nonumber \\
								&=& (-1)^N q^{-N-1}\left(\{N\}!+\frac{3\{N\}!}{1-q^{-1}}\right).
\end{eqnarray}

Note that
\begin{eqnarray}
 (-1)^N\{N\}!&=&(-1)^N\prod_{i=1}^N (q^{-i/2}-q^{i/2})\nonumber \\
 &=&(-1)^N q^{1/2+2/2+\cdots+N/2}\prod_{i=1}^N (q^{-i}-1) \nonumber \\
  &=& q^{\frac{N(N+1)}{4}}\prod_{i=1}^N (1-q^{-i})\nonumber \\
  &\stackrel{\cdot \infty}{=}&\prod_{i=1}^N (1-q^{-i}).
 \end{eqnarray}
 
 Thus 
 \begin{eqnarray}
 J'_{N+1,K}(q)-\text{stabilized head} &\stackrel{\cdot 2N+1}{=}& (-1)^N q^{-N-1}\left(\{N\}!+\frac{3\{N\}!}{1-q^{-1}}\right) \nonumber \\
 &\stackrel{\cdot \infty}{=}&\prod_{i=1}^N (1-q^{-i}) +\frac{3\prod_{i=1}^N (1-q^{-i})}{1-q^{-1}}.
 \end{eqnarray}
This tells us that the tailneck, $T_1$, is the pentagonal numbers plus 3 times the partial sum of the pentagonal numbers.

\subsection{When at least one of the $m_i$ is 1}

When we have an $m_i$ which is 1, we need to consider the $j_i=N-1$ term as well as the $j_i=N$ term.  We can only allow this for the $i$ with $m_i=1$ and only one can be $N-1$ at a time.  Thus we need to determine what this term contributes to the highest $2N+1$ terms and then add it once for each of the $m_i=1$. Label the edge that we will allow to be either $N$ or $N-1$ as $j_1$ and thus we label the edges so $m_1=1$

 Because the degree decreases by $N+1$ when $j_1$ decreases from $N$ to $N-1$ we only need to consider the highest $N$ terms of the $j_1=N-1$ graph evaluation. Call the $j_1=N-1, j_2=j_3=N$ summand $S_{N-1, N, N}$.
 \begin{eqnarray}
 S_{N-1, N, N} 	&=&\prod_{i=1}^3\gamma(N,N,2j_i)\frac{\Delta_{2j_i}}{\theta(N,N,2j_i)}\Gamma(N+1, N-1, N-1) \nonumber \\
 			&=&\underbrace{\gamma(N,N,2N)^2\gamma(N,N,2N-2)}_{\text{shift, does not affect coefficients}}\underbrace{\left(\frac{\Delta_{2N}}{\theta(N,N,2N)}\right)^2}_{=1} \nonumber \\
 			&&\cdot\frac{\Delta_{2N-2}}{\theta(N,N,2N-2)}\Gamma(N+1, N-1, N-1) \nonumber \\
 			&\stackrel{\cdot \infty}{=}&\frac{\Delta_{2N-2}}{\theta(N,N,2N-2)}\Gamma(N+1, N-1, N-1).
  \end{eqnarray}
  We know that
   \begin{eqnarray}
   \theta(N,N,2N-2)	&=&\frac{\Delta_{2N-1}!\Delta_{N-2}!\Delta_{N-2}!}{\Delta_{N-1}!\Delta_{N-1}!\Delta_{2N-3}!}\nonumber \\
   			&=&\frac{\Delta_{2N-1}\Delta_{2N-2}}{\Delta_{N-1}^2}.
   \end{eqnarray}

 So
    \begin{eqnarray}
\frac{\Delta_{2N-2}}{\theta(N,N,2N-2)}	&=&\frac{\Delta_{N-1}^2}{\Delta_{2N-1}} \nonumber\\
						&=&\frac{-\{N\}^2}{\{1\}\{2N\}}.
   \end{eqnarray}
   Also by Lemma \ref{Gammalem} and simplifying the $\Delta_i$ in terms of $\{j\}$ we get:
    \begin{eqnarray}
\Gamma(N+1, N-1, N-1)&=&\frac{(-1)^{N-1}\{3N\}!\{N+1\}!\{N-1\}!^2}{\{1\}\{2N-2\}!\{2N\}!^2}.
  \end{eqnarray}
  Thus we have:
  \begin{eqnarray}
 S_{N-1, N, N} 	&\stackrel{\cdot \infty}{=}&\frac{\Delta_{2N-2}}{\theta(N,N,2N-2)}\Gamma(N+1, N-1, N-1) \nonumber \\
 			&=&\frac{-\{N\}^2}{\{1\}\{2N\}}\frac{(-1)^{N-1}\{3N\}!\{N+1\}!\{N-1\}!^2}{\{1\}\{2N-2\}!\{2N\}!^2} \nonumber \\
 			&=&\frac{(-1)^N\{N\}^2\{3N\}!\{N+1\}!\{N-1\}!^2}{\{1\}^2\{2N-2\}!\{2N\}!^2\{2N\}}.
  \end{eqnarray}
  Let's normalize by dividing by $\frac{(-1)^N\{N+1\}}{\{1\}}$.  Let  $\overline{S_{N-1, N, N}} $ represent the normalized term.  Then we only want the highest $N$ terms so we can reduce $\{N+i\}$ to $-q^{N+i}$.  We can ignore the overall shift that this reduction creates.
    \begin{eqnarray}
 \overline{S_{N-1, N, N}} 	&\stackrel{\cdot \infty}{=}&\frac{\{3N\}!\{N\}!^3}{\{1\}\{2N-2\}!\{2N\}!^2\{2N\}} \nonumber \\
 				&\stackrel{\cdot N}{=}&\frac{(-1)^{2N}\{N\}!\{N\}!^3}{\{1\}(-1)^{N-2}\{N\}!(-1)^{2N}\{N\}!^2(q^N)} \nonumber \\
 				&\stackrel{\cdot \infty}{=}&\frac{(-1)^N\{N\}!^4}{\{1\}\{N\}!^3} \nonumber \\
 				&=&\frac{(-1)^N\{N\}!}{\{1\}}\nonumber \\
				 &\stackrel{\cdot \infty}{=}&\frac{\prod_{i=1}^N (1-q^{-i})}{(q^{-1/2}-q^{1/2})}.\nonumber \\
				 &\stackrel{\cdot \infty}{=}&\frac{-\prod_{i=1}^N (1-q^{-i})}{(1-q^{-1})}.
 	\end{eqnarray}
  This gives us a copy of the pentagonal partial sums for each of the $m_i=1$.  Since the sign here is negative and for the other piece the sign was positive these will cancel with the pentagonal partial sums we got from the other piece. This proves Theorem \ref{thm:triangle}, restated below. \\

  \noindent \textbf{Theorem 1.2.} \emph{Let $m$ be the number edges in the checkerboard graph with $m_i$ of 2 or more.   The neck of knots whose reduced checkboard graph is the triangle graph is:
\begin{equation} \prod_{n=1}^\infty(1-q^{-n})+m\frac{\prod_{n=1}^\infty(1-q^{-n})}{1-q^{-1}}, \end{equation}
i.e. the pentagonal numbers plus the $m$ times the partial sum of the pentagonal numbers.}\\

Using Dasbach's suggestion, we can redefine the neck and tailneck of the colored Jones polynomial by subtracting consecutive terms in the colored Jones sequence, shifted so that they have the same maximum/minimum degree,  instead of subtracting off the stabilized head or tail series. This gives us a simpler expression for the higher order stable pieces. For this class of knots, we have the following: \\

  \noindent \textbf{Corollary 1.3.} \emph{Again, let $m$ be the number edges in the checkerboard graph with $m_i$ of 2 or more.  Then we have
 \begin{equation}J'_{N,K}-q^*J'_{N+1,K}\stackrel{\cdot N-1}{=}(1+m-q^{-1})\prod_{n=1}^\infty(1-q^{-n}). \end{equation}}

In particular, the highest $N$ terms cancel and the next $N-1$ terms agree with the expression above. To see this, recall that the head of the colored Jones polynomial for these knots is given by $\prod_{n=0}^\infty(1-q^{-n}).$  For simplicity, call this expression $P$. Then we have that \begin{equation}J'_{N,K}\stackrel{\cdot 2N-1}{=}\textrm{head} + q^{-N}\textrm{neck}.\end{equation}

Thus, 
\begin{align}
 J'_{N,K}-q^*J'_{N+1,K}&\stackrel{\cdot 2N-1}{=}\left(\textrm{P} + q^{-N}P+\frac{q^{-N}mP}{1-q^{-1}}\right)-\left(\textrm{P} + q^{-N-1}P+\frac{q^{-N-1}mP}{1-q^{-1}}\right) \nonumber \\
 &\stackrel{\cdot N-1}{=}q^{-N}P\left(1-q^{-1}+\frac{m(1-q^{-1})}{1-q^{-1}}\right)\nonumber \\
 &\stackrel{\cdot N-1}{=}P(1+m-q^{-1}) \nonumber \\
 &\stackrel{\cdot N-1}{=}(1+m-q^{-1})\prod_{n=1}^\infty(1-q^{-n}).
 \end{align}


We call this polynomial $T^*_1$.  Using the formula given in the appendix, we can use Mathematica to conjecture the polynomials of the next stable sequence $T^*_2$. For the next sequence, we need to consider the case of the twist numbers ($m_i$) being 1, 2 and 3 or more.  Let $\hat{J}_{N,K}$ be the $N$ colored Jones polynomial shifted so that the highest order term is 1. The polynomials given are $q^{(2N+2)}(q^{-1}(\hat{J'}_{N,K}-\hat{J'}_{N+1,K})-(\hat{J'}_{N+1,K}-\hat{J'}_{N+2,K}))/\prod_{i=1}^\infty(1 - q^{-i})$.  They are listed by listing their coefficients under the appropriate power to best show the observed patterns. These have been checked for various values of $N$ but have not yet been proved. 

\begin{table}[h!]
\caption{This table lists the coefficients for the higher order stable sequences of the colored Jones polynomial of knots whose reduced checkerboard graph is a triangle graph.}
\begin{tabular}{|c|ccc|lccccc|}
\hline
Values of $(m_1,m_2,m_3)$  & $T^*_1:$&1 & $q^{-1}$ &$T^*_2:$&1&$q^{-1}$&$q^{-2}$&$q^{-3}$&$q^{-4}$ \\
(up to permutation)&&&&&&&&&\\
\hline
$(1,1,1)$ & &1&-1 && 0&1&-1&-1&1\\
\hline
$(1,1,2)$& &2&-1 &&0&4&-1&-3&1\\
$(1,1,3^+)$&&2&-1 &&-1&4&0&-3&1\\
\hline
$(1,2,2)$&& 3&-1 &&0&7&0&-4&1\\
$(1,2,3^+)$&&3&-1 &&-1&7&1&-4&1\\
$(1,3^+,3^+)$&&3&-1 &&-2&7&2&-4&1\\
\hline
$(2,2,2)$& &4&-1 &&0&10&2&-4&1\\
$(2,2,3^+)$& &4&-1 &&-1&10&3&-4&1\\
$(2,3^+,3^+)$& &4&-1 &&-2&10&4&-4&1\\
$(3^+,3^+,3^+)$& &4&-1 &&-3&10&5&-4&1\\
\hline
\end{tabular}
\end{table}

There seems to be a lot of information encoded in these coefficients.  One thing to notice is that for any knot with three or more twists in a region, the polynomial starts with a constant term.  This is equivalent to having a $T_1$ sequence whose first $N$ terms match those of $J'_{N+1,K}$ instead of having the first $N+1$ terms match. 

In the future, we hope to prove these conjectures and determine what, if any, effect the graph properties have on the coefficients. We also hope to find similar sequences for other families of knots and even higher order stability.

\appendix
\section{Finding the first $3N+1$ coefficients of the $J_{N+1,K}$ of Certain Knots}
In order to study the further stabilization of knots whose B-checkerboard graph is a triangle graph, we present a formula that allows Mathematica to easily find the highest $3N+1$ coefficients of the $N+1$ colored Jones polynomial of these knots. To develop this formula, we again consider which terms will contribute. 

For notational ease, let's consider $J_{N+1,K}$.  We want the highest $3N+1$ coefficients. This gives us the $N+1$ terms that stabilize in the head as well as the next two sequences which stabilize with length $N$.  In the table below, we list out the possible labels (ordered from highest label to lowest) and the decrease of maximum degree from the $(N,N,N)$ labeling.  These are lower bounds on the decrease in maximum degree found by considering the basic decreases and not whether the number of circles is increased or decreased in each case. They are grouped by the number of changes from the $(N,N,N)$ labeling.

 \begin{table}[h!]
 \caption{This table lists out possible labelings and the decrease maximum degree with that labeling.}
\begin{tabular}{|c|c|}
\hline
Labeling (up to permutation) & decrease in maximum degree from $(N,N,N)$ \\
\hline
$(N,N,N)$ & 0\\
\hline
$(N,N,N-1)$&at least $N+1$\\
\hline
$(N,N,N-2)$&at least $2N+1$\\
$(N,N-1,N-1)$&at least $2N+2$\\
\hline
$(N,N,N-3)$&at least $3N-1$\\
$(N,N-1,N-2)$&at least $3N+1$\\
$(N-1,N-1,N-1)$&at least $3N+3$\\
\hline
\end{tabular}
\end{table}

As we continue to decrease the labelings, the maximum degrees continue to decrease and thus the only terms that contribute to the terms we want to find are the ones with the labelings of $(N,N,N), (N,N,N-1), (N,N,N-2), (N,N-1,N-1)$ and $(N,N,N-3)$. 

To find the formula, we need to determine the evaluation of the graphs with these labelings. The case of the graph labeled with $(N,N,N-i)$ works just like the case of $(N,N,N-1)$ explained in Section \ref{trianglegraph}. And thus, the evaluation of $\Gamma_{N,(N,N,N-i)}=\Gamma(N+i,N-i,N-i)$.

\begin{figure}[ht]
\centering
\begin{subfigure}[b]{.4\textwidth}
\centering
\includegraphics[width=\textwidth]{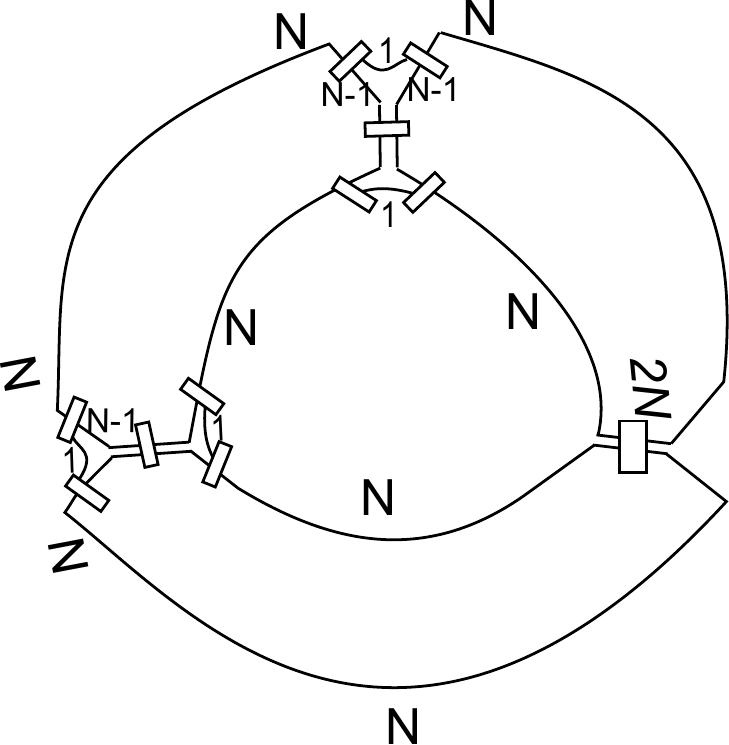}
\caption{The Graph $\Gamma_{N,(N,N-1,N-1)}$}
\end{subfigure}
\qquad
\begin{subfigure}[b]{.4\textwidth}
\centering
\includegraphics[width=\textwidth]{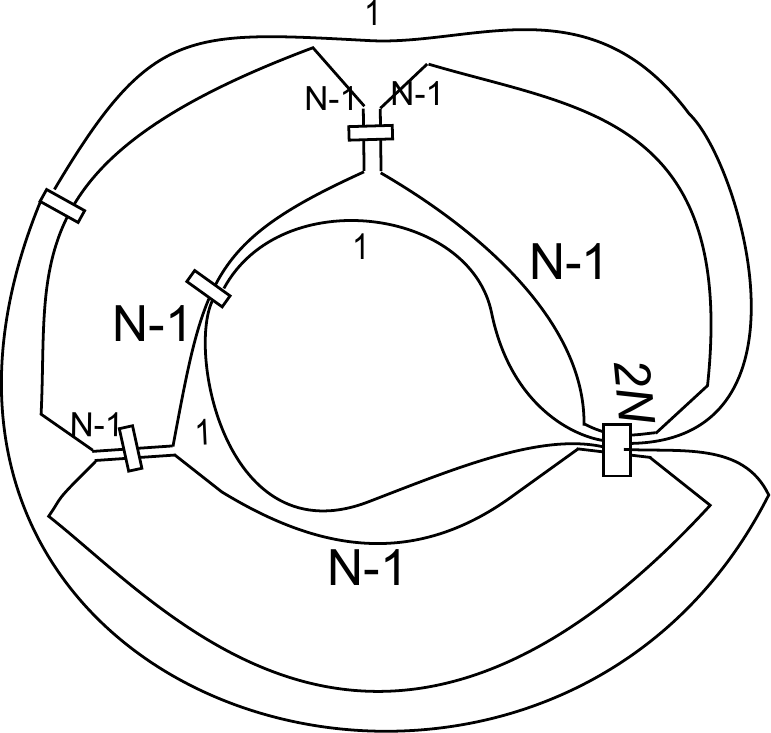}
\caption{We can absorb the smaller idempotents into the larger ones and combine adjoining ones.}
\end{subfigure}

\begin{subfigure}[b]{.4\textwidth}
\centering
\includegraphics[width=\textwidth]{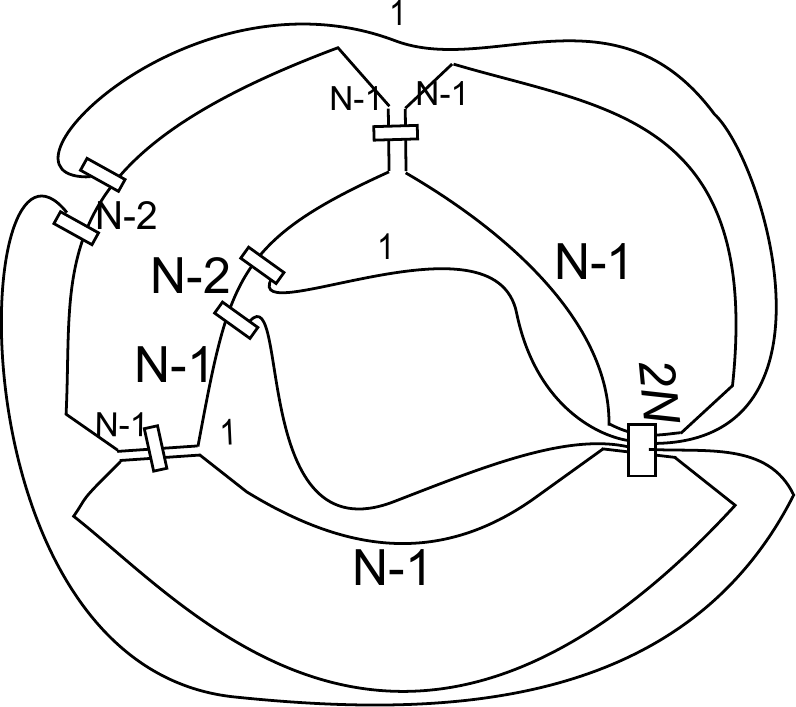}
 \caption{We resolve the remaining $N$ idempotents but only one of the terms is non-zero. We get a factor of $-\left(\frac{\Delta_{N-2}}{\Delta_{N-1}}\right)$ from each resolution. }
\end{subfigure}
\qquad
\begin{subfigure}[b]{.4\textwidth}
\centering
\includegraphics[width=\textwidth]{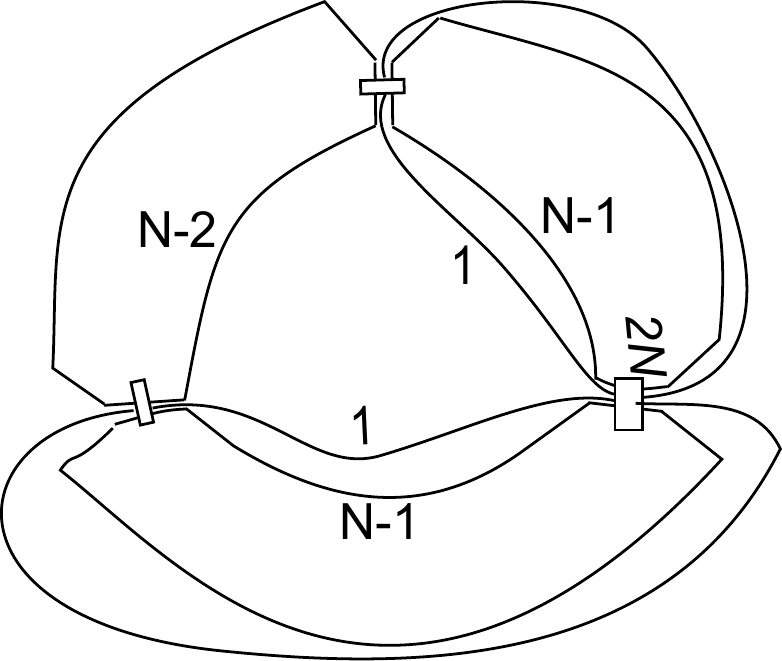}
 \caption{We can then absorb the remaining $N-1$ idempotents.}
\end{subfigure}
 \caption{The evaluation of $\Gamma_{N,(N,N-1,N-1)}=\left(\frac{\Delta_{N-2}}{\Delta_{N-1}}\right)^2\Gamma(N,N,N-2)$}
\label{nnm1nm1reduction}
\end{figure}

The interesting case is the case of $\Gamma_{N,(N,N-1,N-1)}$, see Fig. \ref{nnm1nm1reduction}.  First, we can see that most of the idempotents are absorbed into the larger ones. This leaves four $N$ idempotents which can be combined in pairs. We begin to resolve the remaining two idempotents labeled $N$. The first term in the resolution gives a back track into the $2N$ idempotent for each one. Thus neither of these terms for either idempotent contribute. The other term is non-zero.  We can see that then new $N-1$ labeled idempotents can now be absorbed into the $2N-2$ idempotents. This leaves us with $N-2$ circles going around the top left area and $N$ in each of the other two. We get a $-\frac{\Delta_{N-2}}{\Delta_{N-1}}$ factor from each resolution and thus 
\begin{equation}\Gamma_{N,(N,N-1,N-1)}=\left(\frac{\Delta_{N-2}}{\Delta_{N-1}}\right)^2\Gamma(N,N,N-2).\end{equation}

Now, we can go back to our original expression for the Colored Jones polynomials of knots that reduce to a triangle graph. In particular, 
\begin{equation} J_{N+1,K}=\sum_{j_1,\ldots j_3=0}^{N}\prod_{i=1}^{3}\gamma(N,N,2j_i)^{m_i}\frac{\Delta_{2j_i}}{\theta(N,N,2j_i)}\Gamma_{N,(j_1,\ldots ,j_3)}.\end{equation}
We only need to evaluate the terms indicated above. For each of these terms, we know how to evaluate the graph.  Thus we can evaluate the highest $3N+1$ terms of $J_{N+1,K}(q).$

\bibliographystyle{plain}  
\bibliography{biblio}

\end{document}